\documentclass[12pt]{amsart}

\usepackage[english]{babel}
\usepackage[utf8x]{inputenc}
\usepackage[T1]{fontenc}

\usepackage[alphabetic]{amsrefs}

\usepackage{amsmath}
\usepackage{amstext}
\usepackage{amsfonts}
\usepackage{amssymb}
\usepackage{amsthm}
\usepackage{amsrefs}
\usepackage{mathtools}
\usepackage{dsfont}
\usepackage{color}
\usepackage{graphicx}
\usepackage[colorinlistoftodos]{todonotes}
\usepackage[colorlinks=true, allcolors=blue]{hyperref}
\usepackage{float}
\allowdisplaybreaks

\newtheorem{theorem}{Theorem}[section]
\newtheorem{lemma}[theorem]{Lemma}
\newtheorem{prop}[theorem]{Proposition}

\newtheorem{cor}[theorem]{Corollary}
\newtheorem{thm}[theorem]{Theorem}
\newtheorem{lem}[theorem]{Lemma}

\newtheorem*{cor*}{Corollary}
\newtheorem*{thm*}{Theorem}
\newtheorem*{lem*}{Lemma}

\newtheorem*{prop*}{Proposition}

\theoremstyle{definition}

\newtheorem{defn}[theorem]{Definition}
\newtheorem{example}[theorem]{Example}

\newtheorem*{defn*}{Definition}

\newcommand{\define}[1]{{\it #1}}
\theoremstyle{remark}

\newcommand{\pr}{\operatorname{Prob}}

\newcommand{\act}{\curvearrowright}
\newcommand{\Ad}{\operatorname{Ad}}

\newcommand{\cA}{\mathcal{A}}
\newcommand{\cB}{\mathcal{B}}

\newcommand{\cG}{\mathcal{G}}
\newcommand{\cH}{\mathcal{H}}
\newcommand{\cI}{\mathcal{I}}

\newcommand{\cP}{\mathcal{P}}

\newcommand{\cS}{\mathcal{S}}
\newcommand{\cT}{\mathcal{T}}
\newcommand{\cU}{\mathcal{U}}
\newcommand{\cZ}{\mathcal{Z}}
\newcommand{\cV}{\mathcal{V}}

\newcommand{\bC}{{\mathbb{C}}}
\newcommand{\bE}{{\mathbb{E}}}
\newcommand{\bF}{{\mathbb{F}}}

\newcommand{\bN}{{\mathbb{N}}}
\newcommand{\bZ}{{\mathbb{Z}}}

\newcommand{\bQ}{{\mathbb{Q}}}

\newcommand{\csr}{C^{*}_{\lambda}(\Gamma)}

\renewcommand{\ll}{\ell^\infty}
\newcommand{\aut}{\operatorname{{\bf aut}}}

\newcommand{\im}{\mathrm{Im}}
\newcommand{\ind}{\mathrm{Ind}}

\newcommand{\id}{\operatorname{id}}
\newcommand{\Rad}{\operatorname{Rad}}
\newcommand{\stab}{\operatorname{Stab}}

\newcommand{\Hil}{\mathcal H}

\newcommand{\C}{\mathbb C}
\newcommand{\G}{\Gamma}

\title[]{Topological boundaries of unitary representations}

\author[A. Bearden]{Alex Bearden}
\address{Department of Mathematics\\ University of Texas at Tyler\\USA}
\email{cbearden@uttyler.edu}

\author[M. Kalantar]{Mehrdad Kalantar}
\address{Department of Mathematics\\ University of Houston\\USA}
\email{kalantar@math.uh.edu}

\date{}

\begin{document}

\begin{abstract}
We introduce and study a generalization of the notion of the Furstenberg boundary of a discrete group $\Gamma$ to the setting of a general unitary representation $\pi: \Gamma \to B(\cH_\pi)$. This space, which we call the ``Furstenberg-Hamana boundary'' of the pair $(\Gamma, \pi)$ is a $\Gamma$-invariant subspace of $B(\cH_\pi)$ that carries a canonical $C^*$-algebra structure. In many natural cases, including when $\pi$ is a quasi-regular representation, the Furstenberg-Hamana boundary of $\pi$ is commutative, but can be non-commutative in general. We study various properties of this boundary, and give some applications.
\end{abstract}

\thanks{MK was partially supported by the NSF Grant DMS-1700259.}

\maketitle


\section{Introduction}
The notion of (topological) boundary actions were introduced and studied in the seminal work of Furstenberg \cite{Furs63}, \cite{Furs73}, as a tool to apply in rigidity problems in the context of semisimple groups.
This notion has received much less attention compared to its measurable counterpart in that program, but recently, it has emerged as a powerful tool in several rigidity problems concerning reduced group $C^*$-algebras. More specifically, dynamical properties of the action of a discrete group $\Gamma$ on its Furstenberg boundary $\partial_F\Gamma$ were used to settle problems of simplicity, uniqueness of trace, and tight nuclear embeddings of the reduced $C^*$-algebra $C^*_\lambda(\Gamma)$ of $\Gamma$ (\cite{KK}, \cite{BKKO}, \cite{LeBoud17}).
The connection between dynamical properties of boundary actions and structural properties of the reduced group $C^*$-algebra was initially made through a canonical identification $C(\partial_F\Gamma)\cong \cI_\Gamma(\bC)$, where the latter is the $\Gamma$-injective envelope of the trivial $\Gamma$-$C^*$-algebra in the sense of Hamana \cite{Ham85}.

The main purpose of this paper is to introduce a ``Furstenberg-type boundary'' $\cB_\pi$ associated to a pair $(\Gamma, \pi)$ of a discrete group $\Gamma$ and a unitary representation $\pi: \Gamma\to B(\cH_\pi)$. In view of the above-mentioned identification, this boundary is constructed as a relative version of Hamana's notion of $\Gamma$-injective envelopes, and is naturally identified as a $\Gamma$-invariant subspace $\cB_\pi\subseteq B(\cH_\pi)$. It also admits a canonical $C^*$-algebra structure, as it appears as the image of a u.c.p.\ idempotent. 

Since every unital $C^*$-algebra $\cA$ is of the form $C^*_\pi(\Gamma)$ (the $C^*$-algebra generated in $B(\cH_\pi)$ by $\pi(g)$, $g\in\Gamma$), for some discrete group $\Gamma$ and a unitary representation $\pi$ of $\Gamma$, it is natural to seek for a similar object to the Furstenberg boundary that is associated to a pair $(\Gamma, \pi)$.
We prove some general properties of $\cB_\pi$ which confirm that our definition is natural.

The first problem in the $C^*$-algebra context that the theory of boundary actions was successfully applied to is Ozawa's nuclear embedding conjecture. Recall that by a result of Kirchberg \cite{Kir95}, every exact $C^*$-algebra can be embedded into a nuclear $C^*$-algebra. Although concrete nuclear embeddings have been constructed for exact $C^*$-algebras in certain special cases, in general the nuclear embeddings that are guaranteed to exist by the general abstract results can be difficult to realize. In \cite{Oza07} Ozawa considered the problem of constructing a ``tight'' nuclear embedding of a given exact $C^*$-algebra, and conjectured that for any exact $C^*$-algebra $\cA$, there is a nuclear $C^*$-algebra $\cB$ such that $\cA \subseteq \cB \subseteq \cI(\cA)$, where $\cI(\cA)$ denotes the injective envelope of $\cA$.
He proved the above for the reduced $C^*$-algebra $C^*_\lambda(\bF_n)$ of the free group $\bF_n$, which is known to be exact. In \cite{KK}, Kennedy and the second-named author proved the conjecture above for the reduced $C^*$-algebra of every discrete exact group $\Gamma$ by showing that the $C^*$-algebra $\cB$ generated by $C^*_\lambda(\Gamma)$ and $C(\partial_F\Gamma)$ in $B(\ell^2(\Gamma))$ is nuclear, and $C^*_\lambda(\Gamma) \subseteq \cB \subseteq \cI(C^*_\lambda(\Gamma))$.
Ozawa's conjecture remains open in general. 

Other, even more high-profile, problems in which the application of boundary actions resulted in significance progress were the problems of characterizing discrete groups $\Gamma$ whose reduced $C^*$-algebra $C^*_\lambda(\Gamma)$ is simple (meaning that the only norm-closed two-sided ideals in $C^*_\lambda(\Gamma)$ are zero and $C^*_\lambda(\Gamma)$ itself) or admits a unique trace (namely the canonical trace $\tau_\lambda$ defined by $\tau_\lambda(a):=\langle a\delta_e,\delta_e\rangle$ for $a \in C^*_\lambda(\Gamma)$). 
The group $\Gamma$ is said to be {\em $C^*$-simple} if $C^*_\lambda(\Gamma)$ is simple, and is said to have the {\em unique trace property} if $C^*_\lambda(\Gamma)$ has a unique trace.

Since Powers' proof \cite{P1975} in 1975 that the free group on two generators is both $C^*$-simple and has the unique trace property, it had been a major open problem to characterize groups with either of these properties, and in particular to determine whether they are equivalent (see, e.g., \cite{Del07} for this fact, and for a nice general survey of the subject matter).
In a series of breakthrough works \cite{KK}, \cite{BKKO}, \cite{LeBoud17}, dynamical properties of the Furstenberg boundary $\partial_F\Gamma$ were used to give answers to all the above problems. In \cite{KK}, Kennedy and the second named author proved that a discrete group $\Gamma$ is $C^*$-simple if and only if its action on the Furstenberg boundary $\partial_F \Gamma$ is free. This characterization was used to prove $C^*$-simplicity of a large class of groups, including all previously known examples. Then, in \cite{BKKO}, Breuillard, Kennedy, Ozawa, and the second named author proved that $\Gamma$ has the unique trace property if and only if its action on the Furstenberg boundary $\partial_F \Gamma$ is faithful. In particular, every $C^*$-simple group has the unique trace property. Finally, Le Boudec showed in \cite{LeBoud17} that there are groups $\Gamma$ whose action on the Furstenberg boundary $\partial_F \Gamma$ is faithful but not free.

After proving several properties of the boundary associated to a general representation, we focus our attention to the case of quasi-regular representations. 
It turns out that the associated boundaries of these representations have in fact more tractable and meaningful structures (obviously the class of unital $C^*$-algebras is too large to expect that a general theory of boundaries associated to group representations would be powerful enough to lead to strong results similar to the ones mentioned above in this generality).
We see that in the case of quasi-regular representations $\lambda_{\Gamma/\Lambda}$ associated to subgroups $\Lambda\subset\Gamma$, the boundary $\cB_{\lambda_{\Gamma/\Lambda}}$ is a commutative $C^*$-algebra, hence of the form $\cB_\pi\cong C(X)$ for some compact $\Gamma$-space $X$. We will give a dynamical characterization of this boundary, which should be considered as ``boundary of the quotient''.

In extreme contrast to the special case of the above, we will show that $\cB_\pi$ can be very non-commutative. In fact, we prove that any separable unital purely infinite $C^*$-algebra can be embedded into $\cB_\pi$ for some unitary representation $\pi$ of some discrete countable group $\Gamma$.

In light of the above-mentioned results, one can say that the level of non-triviality of the Furstenberg boundary action $\Gamma\act\partial_F\Gamma$ (e.g. faithfulness or freeness) measures how non-amenable the group $\Gamma$ is.
Recall that a discrete group $\Gamma$ is amenable if there is an invariant mean on the algebra $\ell^\infty(\Gamma)$, i.e., a unital positive linear map $\phi : \ell^\infty(\Gamma) \to \bC$ that is invariant with respect to the action of $\Gamma$ on $\ell^\infty(\Gamma)$ by left translation. By identifying the algebra $\bC$ with a subalgebra of $\ell^\infty(\Gamma)$, the map $\phi$ can be viewed as a unital positive $\Gamma$-equivariant projection. From this perspective, a group $\Gamma$ is non-amenable if the algebra $\bC$ is ``too small'', to admit such a projection. In fact $C(\partial_F\Gamma)$ is the smallest subspace on which $\ell^\infty(\Gamma)$ can be mapped via a unital positive $\Gamma$-equivariant projection.

Observe that since there is a unital positive $\Gamma$-equivariant projection $\psi: B(\ell^2(\Gamma))\to\ll(\Gamma)$ (projection onto the diagonal), where $\Gamma\act B(\ell^2(\Gamma))$ by inner automorphisms via left regular representation, we see that $\Gamma$ is amenable if there is a unital positive $\Gamma$-equivariant projection $\phi : B(\ell^2(\Gamma)) \to \bC$. Motivated by this observation, Bekka 
defined and studied in \cite{Bekka} the notion of amenability for a unitary representation (in the context of locally compact groups). A unitary representation $\pi:\Gamma\to\cU(\cH_\pi)$ is  amenable if there is an invariant mean (or equivalently, a unital positive $\Gamma$-equivariant projection) $\phi: B(\Hil_\pi)\to \bC$.
This notion has become an important concept in the representation theory of locally compact groups. Thus, similarly to the above, a representation $\pi$ is non-amenable if $\bC$ is ``too small'', to admit such a projection. Indeed $\cB_\pi$ is the minimal subspace of $B(\ell^2(\Gamma))$ on which the latter can be mapped by means of a unital completely positive $\Gamma$-equivariant projection. In particular, the level of non-triviality of the action $\Gamma\act \cB_\pi$ is a measure of non-amenability of the representation $\pi$.\\

The organization of the paper is as follows.
In Section 2, we gather some of the preliminary and background material that we will be using throughout the paper.

In Section 3, we construct the boundary $\cB_\pi$ associated to any unitary representation $\pi$ of $\G$, and prove some basic results that show that this boundary generally behaves just as one would expect. We prove that $\cB_\pi$ is completely determined by the weak equivalence class of $\pi$. We also prove some basic results about the boundary in relation to various constructions, including induced representations, tensor products, and restriction to subgroups. We also investigate the situation when $\G$ acts amenably on a FH-boundary.

In Section 4, we study generalized notions of amenability and co-amenability of subgroups relative to a unitary representation, in order to better understand the kernel of the boundary action $\Gamma\act\cB_\pi$.

In Section 5, we prove an extension property for our boundaries, the counterpart of which in the case of the Furstenberg boundary $\partial_F\Gamma$ was the key in \cite[Theorem 1.4]{BKKO} in completely settling the problem of determining the connection between $C^*$-simplicity of a group $\Gamma$ and its normal subgroups, which answered a question of de la Harpe \cite{Del07}.

In Section 6, we consider the problem of uniqueness of certain types of traces and the relation to faithfulness of the boundary action.

In Section 7, we study the boundary of the quasi-regular representation $\lambda_{\Gamma/\Lambda}$ associated to the subgroups $\Lambda$ of $\G$. In this case, $\cB_{\lambda_{\Gamma/\Lambda}}$ is a commutative $C^*$-algebra, hence yields a compact $\G$-space. We prove a dynamical characterization of this compact $\G$-space.

In Section 8, we show that not all boundaries $\cB_\pi$ are commutative.\\

\noindent
\textbf{Acknowledgements.}
The second named author is grateful to Yair Hartman for many fruitful conversations and helpful comments.
We also thanks Nicolas Monod for his thoughtful comments, and pointing out a mistake in an earlier version of this paper.


\section{Preliminaries}

\subsection{Group actions and $C^*$-dynamical systems}

Throughout the paper $\Gamma$ is a discrete group, 
and $\Gamma\act X$ denotes an action of $\Gamma$ by homeomorphisms 
on a compact (Hausdorff) space $X$. In this case we say $X$ is a compact $\Gamma$-space.
The action $\Gamma\act X$ is \define{minimal} if $X$ 
has no non-empty proper closed $\Gamma$-invariant subset.
We denote by $\pr(X)$ the set of all Borel probability measures on $X$.

We will also consider group actions in the measurable setting. An action of $\Gamma$ on a probability space $(X, \nu)$ by measurable isomorphisms is said to be \emph{non-singular} if  $g\nu$ and $\nu$ have the same null sets. Any non-singular action $\Gamma\act (X, \nu)$ canonically induces an action $\Gamma\act L^\infty(X, \nu)$.

For a Hilbert space $\cH$ we denote by $B(\cH)$ the set of all bounded operators on $\cH$. An element $a$ in $B(\cH)$ is said to be positive, written $a\geq 0$, if $a=b^*b$ for some $b\in B(\cH)$. A subalgebra $\cA \subseteq B(\cH)$ is a $C^*$-algebra if it is closed in the operator norm and under taking adjoints. In this case, $\cA$ is \define{unital} if it contains the identity operator on $\cH$.

A linear map $\phi:\cA\to\cB$ between $C^*$-algebras is \define{unital} if $\phi(1_\cA)=1_\cB$, it is \define{positive} if it sends positive elements to positive elements, and it is \define{completely positive} if the maps $\id\otimes \phi : M_n(\bC)\otimes \cA \to M_n(\bC)\otimes \cB$ are positive for all $n\in\bN$. We say $\phi$ is completely isometric if $\id\otimes \phi$ is isometric for all $n\in\bN$. The map $\phi$ is a $*$-homomorphism if it is multiplicative and $\phi(a^*) = \phi(a)^*$ for all $a\in\cA$. A bijective $*$-homomorphism is called a $*$-isomorphism. We denote by $\mathrm{Aut}(\cA)$ the group of all $*$-automorphisms on $\cA$, i.e, $*$-isomorphisms $\cA\to\cA$.

A positive linear functional of norm 1 on a $C^*$-algebra is called a \emph{state}. 

A state $\rho$ is \define{faithful} if $\rho(a)>0$ for any non-zero positive element $a$. A state $\tau$ on a $C^*$-algebra $\cA$ is a \define{trace} if $\tau(ab) = \tau(ba)$ for all $a, b \in \cA$.

A unital $C^{*}$-algebra $\cA$ is called a \define{$\Gamma$-$C^{*}$-algebra} if there is an action $\alpha:\Gamma \act \cA$ of $\Gamma$ on $\cA$ by $*$-automorphisms, that is, $\alpha$ is a group homomorphism $\Gamma \to \mathrm{Aut}(\cA)$.

\subsection{Operator systems}

A subspace $\cV \subseteq B(\cH)$ is an \emph{operator system} if it contains the identity operator on $\cH$, and is closed in the operator norm and under taking adjoints.

The notions of being unital, positive, completely positive, and completely isometric are defined similarly for linear maps between operator systems.

Similarly, an operator system $\cV$ is a \define{$\Gamma$-operator system} if there is an action $\alpha:\Gamma \act \cV$, that is, $\alpha$ is a group homomorphism from $\Gamma$ to the group of all unital bijective completely isometric maps on $\cV$.

We call a linear map $\phi:\cV_1\to \cV_2$ between operator systems a \emph{$\Gamma$-map} if it is unital completely positive (u.c.p.), and $\Gamma$-equivariant, that is
\[
\phi(ga) = g\phi(a), \quad \forall g\in \Gamma, \quad a\in \cV_1 . 
\]
By a \emph{$\Gamma$-embedding} (respectively, \emph{$\Gamma$-projection}) we mean a completely isometric (respectively, idempotent) $\Gamma$-map. A bijective completely isometric $\Gamma$-map is called a \emph{$\Gamma$-isomorphism}. We recall the fact that any bijective completely isometric map between $C^*$-algebras is automatically a $*$-isomorphism of $C^*$-algebras.

Let $\cA$ be a $C^*$-algebra and $\cV\subseteq \cA$ be an operator subsystem. Suppose there is a surjective idempotent unital completely positive map $\psi :\cA \to \cV$. Then by a result of Choi and Effros \cite{ChoEff}, the formula
\[
a \cdot b = \psi(ab), \quad a, b \in \cV .
\]
defines a product on $\cV$, called the \emph{Choi-Effros product}, which turns $\cV$ into a $C^*$-algebra. The $C^*$-algebra obtained in this way is unique up to isomorphism, and in particular does not depend on the map $\psi$.

Any action $\Gamma\act \cV$ induces an adjoint action of $\Gamma$ on the weak*-compact convex set of states on $\cA$. For a state $\nu$ on a $\Gamma$-$C^{*}$-algebra $\cA$ or a $\Gamma$-operator system $\cV$, we denote by $\cP_\nu$ its corresponding \define{Poisson map}, i.e., the unital positive $\Gamma$-equivariant map from $\cA$ or $\cV$ to $\ell^\infty(\Gamma)$ defined by
\begin{equation}\label{def:Poisson-map}
\cP_\nu(a)(g) =  \nu(g^{-1}a), \quad g \in \Gamma .
\end{equation}

\subsection{Unitary representations}
Let $\pi:\Gamma \to \cU(\cH_\pi)$ be a unitary representation on the Hilbert space $\cH_\pi$ (where $\cU(\cH_\pi)$ is the group of unitary operators on $\cH_\pi$). We denote by $C^{*}_{\pi}(\Gamma) := \overline{\operatorname{span}\{\pi(g) : g\in\Gamma\}}^{\|\cdot\|}\subseteq \cB (\cH_\pi)$ the $C^*$-algebra generated by $\pi(\Gamma)$.

In this work we are interested in the dynamics of the action $\Gamma$ on $B(\cH_\pi)$ by inner automorphisms $\Ad_g(x) := \pi(g)x\pi(g^{-1})$, $g\in \Gamma$, $x\in B(\cH_\pi)$, as well as on $C^*$-subalgebras $\cA\subseteq B(\cH_\pi)$ and operator subsystems $\cV\subseteq B(\cH_\pi)$ that are invariant under this action. 

An important class of unitary representations is the (left) quasi-regular representations $\lambda_{\Gamma/\Lambda} : \Gamma \to \cU(\ell^2(\Gamma/\Lambda))$ defined by
\[
(\lambda_{\Gamma/\Lambda}(g)\xi)(h\Lambda) = \xi(g^{-1}h\Lambda)~~~~~~ \left(h\in\Gamma, ~\xi\in \ell^2(\Gamma/\Lambda)\right),\] 
where $\Lambda \leq \Gamma$ is a subgroup. In the case of the trivial subgroup $\Lambda=\{e\}$, the $C^*$-algebra $C^{*}_{\lambda_\Gamma}(\Gamma)$ is called the \define{reduced $C^*$-algebra of $\Gamma$}. And for the choice of $\Lambda= \Gamma$ the corresponding quasi-regular representation is the trivial representation of $\Gamma$, which we denote by $1_\Gamma$.

Let $\pi$ and $\sigma$ be two unitary representations of $\Gamma$. We say $\pi$ is weakly contained in $\sigma$, written $\pi\prec\sigma$, if the map $\sigma(g)\mapsto\pi(g)$ extends to a $*$-homomorphism $C^*_\sigma(\Gamma)\to C^*_\pi(\Gamma)$, which then obviously must be surjective.

Let $\Lambda\leq \Gamma$ be a subgroup, and let $\pi:\Lambda\to\cU(\cH_\pi)$ be a unitary representation of $\Lambda$. We denote by $\ind_\Lambda^\Gamma(\pi)$ the unitary representation of $\Gamma$ induced by $\pi$ (see, e.g., \cite[Appendix E]{BHV08}).

\subsection{Amenability and coamenability}
Recall that a discrete group $\Gamma$ is amenable if $\ell^\infty(\Gamma)$ admits a (left) translation invariant mean, i.e., a state $\phi:\ell^\infty(\Gamma)\to \bC$ such that $\phi(f_g)=\phi(f)$ for every $f\in \ell^\infty(\Gamma)$ and $g\in \Gamma$, where $f_g(h) = f(g^{-1}h)$, $h\in \Gamma$, is the left translation of $f$ by $g$. With the (left) translation action $\Gamma\act\ell^\infty(\Gamma)$ and the trivial action $\Gamma\act\bC$, an invariant mean is nothing but a $\Gamma$-projection $\ell^\infty(\Gamma)\to \bC$. Moreover, there is a $\Gamma$-projection $\bE:B(\ell^2(\Gamma))\to \ell^\infty(\Gamma)$ defined by $\bE(a)(g) = \langle a\delta_g, \delta_g\rangle$ for $a\in B(\ell^2(\Gamma))$ and $g\in\Gamma$, where $\Gamma$ acts on $B(\ell^2(\Gamma))$ by inner automorphisms $\Ad_{\lambda_\Gamma(g)}$. Thus, a group $\Gamma$ is amenable iff there is a $\Gamma$-projection $\bE:B(\ell^2(\Gamma))\to\bC$. 

\begin{defn}\cite{Bekka}
A unitary representation $\pi:\Gamma\to\cU(\cH_\pi)$ is an \emph{amenable representation} if there is a $\Gamma$-projection $\phi: B(\Hil_\pi)\to \bC$, where $\Gamma$ acts on $B(\cH_\pi)$ by inner automorphisms $\Ad_{\pi(g)}$, $g\in\Gamma$.
\end{defn}

By the above comments we see $\Gamma$ is amenable iff the regular representation $\lambda_\Gamma$ is amenable.

Amenability can also be characterized in terms of weak containment. The group $\Gamma$ is amenable iff $1_\Gamma\prec\lambda_\Gamma$, iff $\pi\prec\lambda_\Gamma$ for every unitary representation $\pi$ of $\Gamma$. A subgroup $\Lambda\subseteq \Gamma$ is amenable iff $\lambda_{\Gamma/\Lambda}\prec\lambda_\Gamma$.

Generally, every group $\Gamma$ has a largest amenable normal subgroup, called the \emph{amenable radical} $\Rad(\Gamma)$. It was proved by Furman in \cite{Furm} that $\Rad(\Gamma)$ coincides with the kernel of the action of $\Gamma$ on its Furstenberg boundary $\partial_F \Gamma$.

Let $\Lambda\subseteq \Gamma$ be a subgroup. We say $\Lambda$ is \emph{co-amenable} in $\Gamma$ if $1_\Gamma \prec \lambda_{\Gamma/\Lambda}$, which is equivalent to existence of a $\Gamma$-invariant state (or equivalently, a $\Gamma$-projection) $\ell^\infty(\Gamma/\Lambda)\to\bC$.


\section{The Furstenberg--Hamana boundary: definition and general properties}


\subsection{Injective envelopes} \label{sec:G-inj-envelope}

In this section we recall fundamental notions in the theory of injective envelopes for objects in a category of operator systems. For more details on the general theory of operator systems and injective envelopes, we refer the reader to Hamana's papers \cite{Ham79a,Ham79b}, or to Paulsen's book \cite{Pau02}*{Chapter 15}.

Let $\mathfrak{C}$ be a category of operator systems (with possibly extra structures) whose morphisms are u.c.p.\ maps (again, with possibly more properties). An operator system $\cV$ in $\mathfrak{C}$ is said to be \emph{injective} in $\mathfrak{C}$ if for every completely isometric morphism $\iota : \cS \to \cT$ and every morphism $\psi : \cS \to \cV$ there is a morphisms $\hat{\phi} : \cT \to \cV$ such that $\hat{\phi} \iota = \phi$.

In particular, we are interested the category of operator systems or the category of unital $C^*$-algebras with u.c.p.\ maps as the morphisms, and the category of $\Gamma$-operator systems or unital $\Gamma$-$C^*$-algebras with $\Gamma$-maps as morphisms.

Hamana proves in \cite{Ham85}*{Lemma 2.2} that if $\cV$ is an injective operator system, then the $\Gamma$-operator system $\ell^\infty(\Gamma,\cV)$ is always $\Gamma$-injective.

Let $\cV$ be a ($\Gamma$-)operator systems. A ($\Gamma$-)\emph{extension of $\cV$} is a pair $(\cT, \iota)$ consisting of a ($\Gamma$-)operator system $\cT$, and a completely isometric morphism $\iota : \cV \to \cT$.

A ($\Gamma$-)extension $(\cT, \iota)$ of $\cV$ is ($\Gamma$-)\emph{injective} if $\cT$ is ($\Gamma$-)injective. It is ($\Gamma$-)\emph{essential} if for every morphism $\phi : \cT \to \cS$ such that $\phi \iota$ is completely isometric on $\cV$, $\phi$ is necessarily completely isometric on $\cT$. It is ($\Gamma$-)\emph{rigid} if for every morphism $\phi : \cT \to \cT$ such that $\phi \iota = \iota$ on $\cV$, $\phi$ is necessarily the identity map on $\cT$.

\begin{defn}
Let $\cV$ be a ($\Gamma$-)operator system. A ($\Gamma$-)extension of $\cV$ that is both ($\Gamma$-)injective and ($\Gamma$-)essential is said to be a ($\Gamma$-)\emph{injective envelope of $\cV$}.
\end{defn}

We note that by \cite{Ham85}*{Lemma 2.4}, every $\Gamma$-injective envelope of $\cV$ is $\Gamma$-rigid. 

\begin{thm}[Hamana]\label{thm:existence-G-inj-envelope}
Every operator system $\cV$ has an injective envelope, denoted by $\cI(\cV)$, which is unique up to complete isometric isomorphism. 

Similarly, if $\Gamma$ is a discrete group, then every $\Gamma$-operator system has a $\Gamma$-injective envelope, denoted by $\cI_\Gamma(\cV)$, which is unique up to complete isometric $\Gamma$-isomorphism.
\end{thm}

A unital ($\Gamma$-)$C^*$-algebra is ($\Gamma$-)injective as a $C^*$-algebra if and only if it is ($\Gamma$-)injective as a ($\Gamma$-)operator system.

For any Hilbert space $\cH$ the space $B(\cH)$ is an injective operator system. Thus, if $\cV\subseteq B(\cH)$ is an operator system, then this inclusion extends to a u.c.p.\ map $\cI(\cV)\to B(\cH)$, which is completely isometric by essentiality. If we identify $\cI(\cV)$ with its copy under this embedding, then using injectivity of $\cI(\cV)$ we obtain a u.c.p.\ idempotent from $B(\cH)$ onto $\cI(\cV)$. Hence, the latter turns into a $C^*$-algebra with the Choi-Effros product resulting from the latter u.c.p.\ idempotent.

If $\Gamma$ is a discrete group, then $\ell^\infty(\Gamma,B(\cH))$ is $\Gamma$-injective. Thus, if $\cV\subseteq B(\cH)$ is a $\Gamma$-operator system, then similarly to the above, we have a $\Gamma$-embedding $\cI_\Gamma(\cV)\to \ell^\infty(\Gamma,B(\cH))$, and a $\Gamma$-projection $\ell^\infty(\Gamma,B(\cH))\to \cI_\Gamma(\cV)$, which then provides the latter with a $C^*$-algebra structure via the Choi-Effros product.

\subsection{Existence of relative $\Gamma$-injective envelopes}
Fix a unitary representation $\pi:\Gamma\to\cU(\cH_\pi)$ of $\Gamma$ 
on a Hilbert space $\cH_\pi$, and consider $B(\cH_\pi)$ as a $\G$-$C^*$-algebra via the $\G$-action $s.T = \pi(s)T\pi(s^{-1})$ for $s \in \G$, $T \in B(\cH_\pi)$.
Following similar constructions as in \cite{Ham85} and \cite{Pau11}, 
we prove the existence of the relative $\Gamma$-injective envelope of 
$\bC1 \subseteq B(\cH_\pi)$.

\begin{prop}
On the set $\mathcal G$ of all $\Gamma$-maps $\phi: B(\cH_\pi) \to B(\cH_\pi)$ define the partial pre-order
\[
\phi \leq \psi \, \text{ if } \, \|\phi(x)\| \leq \|\psi(x)\|\, \text{ for all }\, x\in B(\cH_\pi) .
\]
Then $\mathcal G$ contains a minimal element.
\end{prop}

\begin{proof}
We show that every decreasing net $(\phi_i)$ in $\mathcal G$ has a lower bound. Then the result follows from the Zorn's lemma.

To show this, first note that $\mathcal G$ is point-weak* compact.
Therefore, there is a subnet $(\phi_{i_j})$ and $\phi_0\in \mathcal G$
such that $\phi_{i_j}(x) \to \phi_0(x)$ in the weak* topology for all $x\in B(\cH_\pi)$.
We obviously have $\|\phi_0(x)\| \leq \limsup_{i_j} \|\phi_{i_j}(x)\| = \inf_i \|\phi_{i}(x)\|$
for all $x\in B(\cH_\pi)$.
\end{proof}

We denote by $\im(\phi)$ the image of $\phi\in\cG$.

\begin{prop}\label{min-image-properties}
Suppose $\phi_0$ is a minimal element of $\mathcal G$. Then 
\begin{enumerate}
\item[1.]
$\phi_0$ is an idempotent.
\item[2.]
(\emph{$\pi$-essentiality}:) Every $\Gamma$-map $\psi : \im(\phi_0) \to B(\cH_\pi)$ is isometric.
\item[3.]
$\im(\phi_0) \subseteq B(\cH_\pi)$ is minimal among subspaces of $B(\cH_\pi)$ that are images of $\Gamma$-projections. 
\item[4.]
(\emph{$\pi$-rigidity}:) The identity map is the unique $\Gamma$-map on $\im(\phi_0)$.
\item[5.]
(\emph{$\pi$-injectivity}:)
If $X \subseteq Y$ are $\G$-invariant subspaces of $B(\cH_\pi)$ and $\psi: B(\cH_\pi) \to B(\cH_\pi)$ is a $\G$-map such that $\psi(X) \subseteq \im(\phi_0)$, then there is a $\G$-map $\tilde{\psi} : B(\cH_\pi) \to B(\cH_\pi)$ such that $\tilde{\psi}(Y) \subseteq \im(\phi_0)$ and $\tilde{\psi}|_X = \psi|_X$.
\end{enumerate}
\end{prop}

\begin{proof}
1. \ Let $\phi^{(n)} := \frac{1}{n} \sum_{k=1}^n \phi_0^k$. By minimality of $\phi_0$ we have $\|\phi^{(n)}(x)\| = \|\phi_0(x)\|$ for all $x\in B(\cH_\pi)$. Then for any $x \in B(\cH_\pi)$ and $n \in \mathbb N$,
\begin{equation*}
\begin{split}
\|\phi_0(x) - \phi_0^2(x)\| &= \|\phi_0(x - \phi_0(x))\| = \|\phi^{(n)}(x-\phi_0(x))\| \\ 
	&= \tfrac{1}{n}\|\phi_0(x)-\phi_0^{n+1}(x)\| \leq \tfrac{2}{n} \|x\|.
\end{split}
\end{equation*}
Since the latter can be made arbitrarily small for any $x \in B(\cH_\pi)$, $\phi_0=\phi_0^2$.

\noindent
2.\ Suppose $\psi : \im(\phi_0) \to B(\cH_\pi)$ is a $\Gamma$-map.
Then by minimality of $\phi_0$, we have $\|x\| = \|\phi_0(x)\| \leq \|\psi\phi_0(x)\|=\|\psi(x)\|$ for all $x\in \im(\phi_0)$.

\noindent
3.\ Suppose $\Phi: B(\cH_\pi) \to B(\cH_\pi)$ is a $\G$-idempotent such that $\im(\Phi) \subseteq \im(\phi_0)$. Then $\Phi|_{\im(\phi_0)}$ is isometric by (2), so for any $x \in B(\cH_\pi)$,
\[
\| \phi_0(x)-\Phi(\phi_0(x))\| = \|\Phi(\phi_0(x)-\Phi(\phi_0(x)))\|=0.
\]
Hence $\phi_0(x) \in \im(\Phi)$ for any $x \in B(\cH_\pi)$, so $\im(\Phi)=\im(\phi_0)$.

\noindent
4.\ Suppose $\Psi: \im(\phi_0) \to \im(\phi_0)$ is a $\G$-map. By (2), $\Psi$ is isometric. For any $x \in B(\cH_\pi)$, $\|\Psi \circ \phi_0(x)\| \leq \|\phi_0(x)\|$. By minimality of $\phi_0$, $\|\Psi \circ \phi_0(\cdot)\| = \|\phi_0(\cdot)\|$. Thus $\Psi \circ \phi_0$ is a minimal element of $\mathcal G$, and so by (1), $\Psi \circ \phi_0$ is idempotent. So for any $x \in B(\cH_\pi)$,
\[\begin{split}
\|\phi_0(x)-\Psi \circ \phi_0(x)\| &= \|\Psi (\phi_0(x)-\Psi \circ \phi_0(x))\| \\&= \|\Psi \circ \phi_0(x)-(\Psi \circ \phi_0)^2(x)\| = 0.
\end{split}\]
Hence $\Psi \circ \phi_0 = \phi_0$, i.e., $\Psi = \mathrm{id}$.

\noindent
5.\ The map $\tilde{\psi} = \phi_0 \circ \psi$ meets the requirements.
\end{proof}

\begin{prop}
The image of a minimal element of $\cG$ is unique up to isomorphism.
\end{prop}

\begin{proof}
Suppose $\phi, \psi:B(\cH_\pi) \to B(\cH_\pi)$ are both minimal elements of $\cG$.
Then by $\pi$-rigidity (4.\ in Proposition \ref{min-image-properties}), the restriction of the composition $\phi \circ \psi$ to $\im(\phi)$ is the identity map. Similarly, the restriction of $\psi \circ \phi$ to $\im(\psi)$ is the identity map. Hence $\psi|_{\im(\phi)}: \im(\phi) \to \im(\psi)$ is a $\Gamma$-isomorphism.
\end{proof}

\begin{defn}
Let $\phi$ be a minimal element of $\mathcal G$, then $\im(\phi)$ equipped with the Choi-Effros product is a $\Gamma$-$C^*$-algebra, which we call the {\it Furstenberg--Hamana boundary} (or FH-boundary) of the representation $\pi$, and will denote by $\cB_\pi$.
\end{defn}

Before getting into properties of the boundary $\cB_\pi$, let us give a few examples.


\begin{example}
We have $\cB_{\lambda_\Gamma} = C(\partial_F \G)$.
\end{example}

\begin{example}
Let $\Gamma\act (X, \nu)$ be a probability measure preserving action, and let $\kappa:\G\to B(L^2(X, \nu))$ be the corresponding Koopman representation. The vector functional associated to the vector $\mathds{1}_X\in L^2(X, \nu)$ is an invariant state on $B(L^2(X, \nu))$, hence $\cB_\kappa$ is trivial.
\end{example}

\begin{example}\label{ex:koop-dis}
Suppose $\G$ acts on a discrete set $S$, considered as a measure space with the counting measure. Let $\lambda_S : \G \to B(\ell^2 (S))$ be the corresponding Koopman representation $\lambda_S(g)(\delta_s) = \delta_{gs}$ for $g \in \G$, $s \in S$. Since the map $\Phi: B(\ell^2 (S)) \to \ell^\infty(S)$, $\Phi(T)(x) = \langle T \delta_x, \delta_x \rangle$, is a $\G$-projection, it follows that $\cB_{\lambda_S}$ may also be characterized as the image of a minimal $\G$-projection on $\ell^\infty(S)$. Hence there is a compact $\G$-space $\partial(\G,S)$ such that $\cB_{\lambda_S} = C(\partial(\G,S))$. 
Observe that $\partial(\G,S)$ is trivial if and only if the action is amenable in the sense of Greenleaf, i.e., there is an invariant mean on $\ell^\infty(S)$.
On the opposite end, we have 
$\partial(\G,S) = \partial_F \G$ if and only if 
$\text{Stab}_\G(s)$ is amenable for each $s \in S$.

To see this, first assume $\partial(\G,S) = \partial_F \G$. Then the inclusion $C(\partial_F \G) \subseteq \ell^\infty(S)$ yields a $\Gamma$-map $\phi:S\to\pr(\partial_F \G)$. Let $\nu\in \pr(\partial_F \G)$, and let $E: \ell^\infty(\Gamma)\to C(\partial_F \G)$ be a $\Gamma$-map. Then $\nu\circ E$ is a $\stab_\Gamma(\nu)$-invariant mean on $\ell^\infty(\Gamma)$, which implies $\stab_\Gamma(\nu)$ is amenable. Hence, for any $s\in S$, we have $\stab_\Gamma(s)\subseteq \stab_\Gamma(\phi(s))$ is amenable. 
Conversely, suppose the point stabilizers of the action $\Gamma\act S$ are all amenable. Then for each $s\in S$ there is a $\Gamma$-map $\phi_s: \ell^\infty(\G) \to\ell^\infty(\Gamma / \stab_\G(s))$. Now let $\{x_i\}$ be a set of representatives of orbit-equivalence classes 
of the action $\Gamma\act S$. Then 

the map 
$\psi=\oplus_i \phi_{x_i} : \ell^\infty(\Gamma)\to \oplus_i \ell^\infty(\Gamma/\stab_\G(x_i)) = \oplus_i \ell^\infty(\Gamma x_i) = \ell^\infty(S)$
is a $\Gamma$-map.
Let $\Phi: \ell^\infty(S) \to C(\partial(\G,S))$ be a $\G$-projection. Then by $\G$-essentiality, $\Phi \circ \psi$ restricts a $\G$-embedding of $C(\partial_F \G)$ into $C(\partial(\G,S))$. By minimality of the latter, this embedding is surjective.

\end{example}

\begin{example}
Let $\Gamma=\bF_n$ be the free group on $n$ generators, $n\geq 2$. Let $\rho_\G:\G\to B(\ell^2(\G))$ be the right regular representation. Let $\pi:\G\to B(\ell^2(\G))$ be the unitary representation $\pi(g):=\lambda_\G(g)\rho_\G(g)$. Observe that the Dirac function $\delta_e$ at the neutral element $e\in\G$ is an invariant vector for $\pi$, thus its corresponding vector functional is an invariant state on $B(\ell^2(\G))$. Hence $\cB_\pi=\bC$ is trivial.

Now, let $\pi_0:\G\to B(\ell^2(\G\setminus\{e\}))$ be the unitary representation obtained from restricting $\pi$ to $\ell^2(\G\setminus\{e\})$, which is the orthogonal complement of the space of invariant vectors of $\pi$. The representation $\pi_0$ is in fact the Koopman representation associated to the action $\G\act S=\G\setminus\{e\}$ by conjugations, which is also unitarily equivalent to the direct sum of quasi-regular representations associated to stabilizer subgroups. Note that for $g\in S$, the stabilizer of $g$ in $\G$ is the centralizer $C_\G(g)$ of $g$ in $\G$, hence amenable. Thus, as seen in Example \ref{ex:koop-dis}, we have $\cB_{\pi_0} = C(\partial_F\G)$.
\end{example}

\begin{example}
Recall that a non-singular action $\Gamma\act (X, \nu)$ on a probability space $(X, \nu)$ is amenable (in the sense of Zimmer) if and only if $L^\infty(X, \nu)$ is $\Gamma$-injective. In this case there is an equivariant embedding $C(\partial_F\Gamma)\subseteq L^\infty(X, \nu)$, and hence also an equivariant projection $L^\infty(X, \nu) \to C(\partial_F\Gamma)$. Hence, by rigidity we have $\cB_\kappa = C(\partial_F\Gamma)$, where $\kappa$ is the Koopman representation associated to the action $\Gamma\act (X, \nu)$.
\end{example}



\subsection{General properties and various constructions}
In this section we study general properties of $\cB_\pi$, especially in regards to various notions and constructions in representation theory, e.g., tensor products, restrictions, inductions and weak containment. These results in particular show that the definition is natural.

\subsubsection{Boundary triviality}
One expects that triviality of the boundary should be equivalent to some amenability property. In fact, the following is obvious from the definitions.

\begin{prop}\label{triv-bnd<-->amen}
Let $\pi$ be a unitary representation of $\Gamma$. Then $\cB_\pi = \bC$ if and only if $\pi$ is an amenable representation.
\end{prop}

\begin{example}
By a result of Bekka and Valette \cite[Theorem 1]{Bekka-Valette} if $\Gamma$ has property~T, then every amenable representation contains a finite dimensional subrepresentation. Therefore, it follows from Proposition \ref{triv-bnd<-->amen} that $\cB_\pi$ is non-trivial for any irreducible infinite dimensional representation $\pi$ of $\Gamma$.
\end{example}

\begin{lemma}\label{stab-rel-amen}
Let $\pi$ be a unitary representation of $\Gamma$, let $\rho$ be a state on $\cB_\pi$, and let $\Lambda = \{g\in\Gamma : g\rho=\rho\}$ be the stabilizer subgroup of $\rho$. Then the restriction of $\pi$ to $\Lambda$ is an amenable representation of $\Lambda$.
\end{lemma}

\begin{proof}
Let $\psi: B(\cH_\pi) \to \cB_\pi$ be a $\Gamma$-projection. Then $\rho\circ\psi$ is a $\Lambda$-invariant state on $B(\cH_\pi)$.
\end{proof}

\begin{cor}
Let $\pi$ be a unitary representation of $\Gamma$, and let $\Lambda\leq \Gamma$ be the kernel of the action $\Gamma\act\cB_\pi$. Then the restriction of $\pi$ to $\Lambda$ is amenable as a representation of $\Lambda$.
\end{cor}

We will study in more details the kernel of the FH-boundary action, but we record the following first observation here, which is another fact that one expects for a natural notion of boundary action.

\begin{prop} \label{finconjclass-trivialaction}
Every $g\in\Gamma$ with finite conjugacy class acts trivially on $\cB_\pi$ for any unitary representation $\pi$ of $\Gamma$.
\end{prop}

\begin{proof}
Let $F\subseteq \Gamma$ be the conjugacy class of $g$. Then $\frac{1}{\#F}\sum_{h\in F}\mathrm{Ad}_{\pi(h)}$ is a $\Gamma$-map on $\cB_\pi$, hence the identity by $\Gamma$-rigidity. Since projections are extreme points of the positive unital ball, it follows that $\pi(h) p \pi(h^{-1}) = p$ for every $h\in F$ and every projection $p\in \cB_\pi$. Since $\cB_\pi$ is an injective $C^*$-algebra, it is generated by its projections. Hence $\mathrm{Ad}_{\pi(h)}$ is the identity map on $\cB_\pi$ for every $h\in F$.
\end{proof}

\subsubsection{Weak containment and weak equivalence}

\begin{prop}\label{prop:weak-contmnt-->map-bnds}
Let $\pi$ and $\sigma$ be two unitary representations 
of $\Gamma$. Then there is a $\Gamma$-map 
from $\cB_\sigma$ to $\cB_\pi$ iff
there is a $\Gamma$-map 
from $B(\cH_\sigma)$ to $B(\cH_\pi)$.
In particular, if $\pi$ is weakly contained in $\sigma$, 
then there is a $\Gamma$-map 
$\cB_\sigma\to\cB_\pi$.
\end{prop}

\begin{proof}
The first part follows immediately from the fact that there are $\Gamma$-maps $\cB_\sigma\to B(\cH_\sigma) \to \cB_\sigma$, and similarly for $\pi$. If $\pi\prec\sigma$, then the canonical $*$-homomorphism $C^*_\sigma (\Gamma)\to C^*_\pi (\Gamma)$ extends to a u.c.p.\ map $B(\cH_\sigma)\to B(\cH_\pi)$ which is automatically $\Gamma$-equivariant since $C^*_\sigma (\Gamma)$ is in its multiplicative domain, and that implies the last assertion.
\end{proof}

\begin{cor} \label{cor:bndy in inj env}
The FH-boundary $\cB_\pi$ of a representation $\pi$ is determined by the weak-equivalence class of $\pi$.
\end{cor}

\begin{proof}
By Proposition \ref{prop:weak-contmnt-->map-bnds}, if $\pi$ and $\sigma$ are weakly equivalent, then there are $\Gamma$-maps $\cB_\pi \to \cB_\sigma$ and $\cB_\sigma \to \cB_\pi$. By rigidity, the composition of these in either order is the identity. Thus $\cB_\pi \cong \cB_\sigma$.
\end{proof}

The following is another justification for the above fact.

\begin{prop} \label{prop:bndy in inj env}
Let $\pi:\Gamma\to \cU(\cH_\pi)$ be a unitary representation of $\Gamma$.
There is a $\Gamma$-inclusion 
$\cB_\pi \subseteq \cI(C^*_\pi (\Gamma))$. 
\end{prop}

\begin{proof}
By injectivity, there is a projection $\phi:B(\cH_\pi) \to \cI(C^*_\pi (\Gamma))$ 
that extends the identity map on $C^*_\pi (\Gamma)$, hence in particular is $\Gamma$-equivariant. Then the restriction of $\phi$ to 
$\cB_\pi$ is an embedding by Proposition \ref{min-image-properties}.
\end{proof}

A natural question now is: To what extent does the FH-boundary of a representation determines its weak-equivalence class? Or more generally (and more reasonably), are there possible relations between FH-boundaries of various weak-equivalence classes of representations? We prove below a partial result in this direction showing that under certain assumptions we may classify the representations whose FH-boundaries coincide with that of the regular representation, namely the Furstenberg boundary of the group. 
\begin{thm}\label{thm:exact--w-cntmnt-reg<-->bnd-Furs}
If $\pi$ is a unitary representation of $\Gamma$ which is weakly contained in the regular representation $\lambda_\Gamma$, then $\cB_\pi \cong C(\partial_F \Gamma)$.

Conversely, if $\Gamma$ is an exact group with $\cB_\pi \cong C(\partial_F \Gamma)$, and
if there is a copy of $\cB_\pi$ that embeds as a $\Gamma$-invariant subspace of a commutative $C^*$-subalgebra of $B(\cH_\pi)$, then $\pi$ is weakly contained in $\lambda_\Gamma$.
\end{thm}

\begin{proof}
For the first assertion, note that since $\pi\prec \lambda_\Gamma$, by Proposition \ref{prop:weak-contmnt-->map-bnds} we have a $\Gamma$-map $\cB_\lambda = C(\partial_F \Gamma) \to \cB_\pi$. By $\Gamma$-injectivity of $C(\partial_F \Gamma)$, there is also a $\Gamma$-map $\cB_\pi \to C(\partial_F \Gamma)$. By rigidity, the latter is in fact an isomorphism.

Conversely, suppose that $\Gamma$ is exact with $\cB_\pi \cong C(\partial_F \Gamma)$ which is embedded as a $\Gamma$-invariant subspace of a commutative $C^*$-subalgebra of $B(\cH_\pi)$. Let $\cA$ be the $C^*$-algebra generated by $\cB_\pi$ in $B(\cH_\pi)$. By the assumptions, we may take $\cA$ to be commutative and $\Gamma$-invariant, hence of the form $C(Y)$ for some compact $\Gamma$-space $Y$. The adjoint to the embedding $C(\partial_F \Gamma)\to C(Y)$ yields a $\Gamma$-equivariant map $Y\to \pr(\partial_F \Gamma)$. Since $\Gamma$ is exact, by \cite[Theorem 4.5]{KK} the action $\Gamma\act \partial_F \Gamma$ is (topologically) amenable, hence so is $\Gamma\act \pr(\partial_F \Gamma)$ (see, e.g., \cite[Lemma 3.6]{H2000}), which implies $\Gamma\act Y$ is amenable. Thus, the full and reduced crossed product of the latter action coincide \cite[Theorem 5.3]{Del02}. Hence, there is a canonical surjective $*$-homomorphism $\Gamma\ltimes_r C(Y)\to C^*(C(Y)\cup \pi(\Gamma))\subseteq B(\cH_\pi)$ which sends $\lambda_\Gamma(g)\mapsto \pi(g)$ for every $g\in \Gamma$. This implies $\pi\prec \lambda_\Gamma$.
\end{proof}


\subsubsection{Restriction to subgroups}

In this section we prove several results concerning  FH-boundaries of restrictions of representations.

\begin{lem}\label{lem:subgrp-bnd-embeds}
Let $\pi$ be a unitary representation of $\Gamma$, and let $\Lambda\leq \Gamma$ be a subgroup. Then there is a $\Lambda$-embedding $\cB_{\pi|_\Lambda} \subseteq \cB_\pi$.
\end{lem}

\begin{proof}
Let $\varphi: B(\cH_\pi)\to \cB_\pi$ be a $\Gamma$-projection. Considering $\varphi$ as a $\Lambda$-map, it follows from Proposition \ref{min-image-properties} that its restriction to $\cB_{\pi|_\Lambda}$ is a $\Lambda$-embedding into $\cB_\pi$.
\end{proof}

\begin{lem}
Let $\pi$ be a representation of $\Gamma$ and let $\Lambda\leq \Gamma$ be a normal subgroup. If the action $\Lambda\act \cB_{\pi|_\Lambda}$ is faithful, then $$\ker(\Gamma\act \cB_{\pi}) \subseteq C_{\Gamma}(\Lambda) ,$$ where $C_{\Gamma}(\Lambda)$ is the centralizer of $\Lambda$ in $\Gamma$.
\end{lem}

\begin{proof}
By Lemma \ref{lem:subgrp-bnd-embeds} we may identify $\cB_{\pi|_\Lambda}$ with a $\Lambda$-invariant subspace of $\cB_\pi$. Now let $g\in\ker(\Gamma\act \cB_\pi)$. Then for every $h\in\Lambda$ and $a\in\cB_{\pi|_\Lambda}$ we have $\pi(g)\pi(h)\pi(g^{-1})a\pi(g)\pi(h^{-1})\pi(g^{-1}) = \pi(h)a\pi(h^{-1})$ which implies $h^{-1}ghg^{-1} \in \ker(\Lambda\act \cB_{\pi|_\Lambda}) = \{e\}$, hence $g\in C_{\Gamma}(\Lambda)$. 
\end{proof}

\begin{cor}
Let $\pi$ be a representation of $\Gamma$, and let $\Lambda\leq \Gamma$ be a normal subgroup with trivial centralizer in $\Gamma$. If the action $\Lambda\act \cB_{\pi|_\Lambda}$ is faithful, then so is the action $\Gamma\act \cB_\pi$. 
\end{cor}

\subsubsection{Tensor products}

\begin{prop}\label{prop:tensor-prod}
Let $\pi_1, \pi_2$ be representations of $\Gamma$. Then there are $\Gamma$-maps $\cB_{\pi_i}\to \cB_{\pi_1\otimes\pi_2}$ for $i=1,2$.
\end{prop}

\begin{proof}
Compose the $\Gamma$-map $B(\cH_{\pi_1}) \to B(\cH_{\pi_1})\otimes B(\cH_{\pi_2})$, $a \mapsto a \otimes I$, with a $\Gamma$-map $B(\cH_{\pi_1})\otimes B(\cH_{\pi_2}) \to \cB_{\pi_1\otimes\pi_2}$, and restrict the resulting map to $\cB_{\pi_1}$ to obtain a $\Gamma$-map $\cB_{\pi_1}\to \cB_{\pi_1\otimes\pi_2}$. The argument for $\pi_2$ is similar.
\end{proof}

\subsubsection{Induction}
In this section we prove some results concerning FH-boundaries of induced representations.

The first result generalizes a special case of \cite[Corollary 5.6]{Bekka}. 

\begin{prop} \label{ind repn prop}
Suppose $\Lambda \leq \G$ is a subgroup, $\pi$ is a unitary representation of $\G$, and $\sigma$ is a unitary representation of $\Lambda$ such that $\pi \prec \mathrm{ind}_\Lambda^\G(\pi|_\Lambda)$ and $\pi|_\Lambda \prec \sigma \otimes \tau$ for some unitary representation $\tau$ of $\Lambda$. Then there is a $\G$-map \( \cB_{\mathrm{ind}_\Lambda^\G \sigma} \to \cB_\pi. \)
\end{prop}

\begin{proof}
By general properties of induced representations and weak containment (see, e.g., \cite[Appendix E, F]{BHV08}), we have
\[\begin{split} 
\pi &\prec \mathrm{ind}_\Lambda^\G(\pi|_\Lambda) \prec \mathrm{ind}_\Lambda^\G(\sigma \otimes \tau) \\&\prec \mathrm{ind}_\Lambda^\G ((\mathrm{ind}_\Lambda^\G \sigma)|_\Lambda \otimes \tau) = \mathrm{ind}_\Lambda^\G \sigma \otimes \mathrm{ind}_\Lambda^\G \tau. 
\end{split}\]
Thus, combining Proposition \ref{prop:tensor-prod} with Proposition \ref{prop:weak-contmnt-->map-bnds} yields the result.
\end{proof}

\begin{thm}
If $\Lambda$ is a co-amenable subgroup of $\G$ and $\pi$ is a unitary representation of $\G$, then there is a $\G$-isomorphism $\cB_{\mathrm{ind}_\Lambda^\G (\pi|_\Lambda)}\cong \cB_\pi$.
\end{thm}

\begin{proof}
Since $\Lambda$ is co-amenable in $\G$, it follows that
\[
\pi = \pi \otimes 1_\G \prec \pi \otimes \lambda_{\G/\Lambda} = \pi \otimes \mathrm{ind}_\Lambda^\G 1_\Lambda = \mathrm{ind}_\Lambda^\G(\pi|_\Lambda \otimes 1_\Lambda) = \mathrm{ind}_\Lambda^\G (\pi|_\Lambda) ,
\]
which implies there is a $\Gamma$-map $\cB_{\mathrm{ind}_\Lambda^\G (\pi|_\Lambda)}\to \cB_\pi$. On the other hand, since $\mathrm{ind}_\Lambda^\G (\pi|_\Lambda)=\pi\otimes \lambda_{\G/\Lambda}$, Proposition \ref {prop:tensor-prod} yields a $\Gamma$-map $\cB_\pi\to \cB_{\mathrm{ind}_\Lambda^\G (\pi|_\Lambda)}$. By $\Gamma$-rigidity both above $\Gamma$-maps are isomorphisms.
\end{proof}


\subsection{Amenability of $\pi$-boundary actions}
In \cite{Oza07} Ozawa conjectured that for any separable unital exact $C^*$-algebra $\cA$ there is a unital nuclear $C^*$-algebra $\cB$ such that $\cA\subseteq \cB\subseteq \cI(\cA)$. The conjecture was proved for the reduced $C^*$-algebras of exact groups in \cite{KK}, where it was shown that if $C^*_\lambda(\Gamma)$ is exact then $\Gamma\ltimes_r C(\partial_F \Gamma)$ is nuclear and $C^*_\lambda(\Gamma)\subseteq\Gamma\ltimes_r C(\partial_F \Gamma)\subseteq\cI(C^*_\lambda(\Gamma))$. The conjecture remains open in general.

Recall from Proposition \ref{prop:bndy in inj env} that $\cB_\pi \subseteq \cI(C^*_\pi(\Gamma))$. Since any separable unital $C^*$-algebra $\cA$ is of the form $C^*_\pi(\Gamma)$ for some countable group $\Gamma$ and a unitary representation $\pi$ of $\Gamma$, by analogy with the result from \cite{KK} just mentioned, it is natural to study properties of $C^*$-algebras generated by $\pi(\Gamma)$ and $\cB_\pi$ in $B(\cH_\pi)$.

Say that an FH-boundary is \emph{$C^*$-embeddable} if there is a $*$-homomorphic copy of $\cB_\pi$ in $B(\cH_\pi)$. Fix such a copy, and let $\tilde \cB_\pi$ be the $C^*$-algebra generated by $\pi(\Gamma)$ and $\cB_\pi$ in $B(\cH_\pi)$. One natural question in light of the above discussion is: in which cases does exactness of $C^*_\pi(\Gamma)$ imply nuclearity of $\tilde\cB_\pi$?

Recall that the nuclearity of the reduced crossed product follows from amenability of the action.

The definition of an amenable group action is due to Anantharaman-Delaroche \cite{Del02}*{Definition 2.1}.
For the general theory of exactness and amenable actions, we refer the reader to the book of Brown and Ozawa \cite{BroOza08}.

In the general case of a $C^*$-embeddable FH-boundary, we should not seek for amenability of the action, as the next theorem shows.

\begin{thm}
Let $\pi$ be a unitary representation of $\Gamma$ such that there is $\cB_\pi$ is $C^*$-embeddable. Then the action $\Gamma\act \cB_\pi$ is topologically amenable if and only if $\Gamma$ is an exact group and $\pi$ is weakly contained in the left regular representation.
\end{thm}

\begin{proof}
Suppose $\Gamma\act \cB_\pi$ is topologically amenable. 
Then the restriction of the action to the center of $\cB_\pi$ is also amenable, hence $\Gamma$ is an exact group by \cite[Theorem 7.2]{Del02}.
Moreover, we have $\Gamma\ltimes\cB_\pi = \Gamma\ltimes_r\cB_\pi$ by \cite[Theorem 5.3]{Del02}. Thus, since $\cB_\pi$ is $C^*$-embeddable, there is a $*$-homomorphism $\Gamma\ltimes_r\cB_\pi \to \tilde\cB_\pi$ that maps $\csr$ onto $C^*_\pi(\Gamma)$ canonically. This implies $\pi \prec \lambda_\Gamma$.

Conversely, suppose $\Gamma$ is exact and $\pi \prec \lambda_\Gamma$. 
By \ref{thm:exact--w-cntmnt-reg<-->bnd-Furs} we have $\cB_\pi\cong C(\partial_F\Gamma)$. Hence, the action $\Gamma\act \cB_\pi$ is topologically amenable by \cite[Theorem 4.5]{KK}.
\end{proof}


\section{Amenability and co-amenability of normal subgroups with respect to unitary representations}

Throughout this section $\Lambda$ is a normal subgroup of $\G$. 
Note that in this case, for any unitary representation $\pi$ of $\G$, 
the commutant $\pi(\Lambda)'$ in $B(\Hil_\pi)$ is $\G$-invariant.

\subsection{$\pi$-amenability}

\begin{defn}
Say that $\Lambda$ is \emph{$\pi$-amenable} if 
there is a $\Gamma$-map $\varphi: B(\Hil_\pi) \to \pi(\Lambda)'$. 
\end{defn}

\begin{prop}\label{pi-amen-equiv-conds}
Let $\pi$ be a unitary representation of $\G$. Then 
the following are equivalent:
\begin{enumerate}
\item 
$\pi$ is an amenable representation;
\item
any normal subgroup $\Lambda$ is $\pi$-amenable;
\item
$\G$ is $\pi$-amenable.
\end{enumerate}
\end{prop}

\begin{proof}
$(1)\Rightarrow (2)$: 
Suppose $\pi$ is an amenable representation, i.e., there is a $\Gamma$-invariant 
mean $\phi$ on $B(\Hil_\pi)$. We may consider $\phi$ as a 
$\Gamma$-map $\varphi: B(\Hil_\pi) \to \pi(\Lambda)'$.
$(2)\Rightarrow (3)$ is trivial. 
$(3)\Rightarrow (1)$: 
suppose $\G$ is $\pi$-amenable, and let 
$\varphi: B(\Hil_\pi) \to \pi(\Gamma)'$ be a $\Gamma$-map. 
Note that $\Gamma$ acts trivially on $\pi(\Gamma)'$.  
Thus, for any state $\rho$ on $\pi(\Gamma)'$, the composition 
$\rho\circ \varphi$ is a $\Gamma$-invariant 
mean on $B(\Hil_\pi)$.
\end{proof}

\begin{prop}
Let $\lambda_\G$ be the left regular representation of $\G$. 
Then a normal subgroup $\Lambda \trianglelefteq \G$ is $\lambda_\G$-amenable 
if and only if it is amenable.
\end{prop}

\proof
Suppose $\Lambda$ is amenable. Then there is a $\Gamma$-map from $\ell^\infty(\Gamma)$ to $\ell^\infty(\Gamma/\Lambda)$. Composing this map with the $\Gamma$-projection $B(\ell^2 (\G))\to \ll(\G)$ gives a $\Gamma$-map $B(\ell^2 (\G))\to \ell^\infty(\Gamma/\Lambda)$. Since $\ell^\infty(\Gamma/\Lambda)\subseteq \pi(\Lambda)'$, it follows $\Lambda$ is $\lambda_\Gamma$-amenable. 

Conversely, suppose $\Lambda$ is $\lambda_\Gamma$-amenable, and suppose $\varphi: B(\Hil_\pi) \to \lambda_\Gamma(\Lambda)'$ is a $\Gamma$-map. Then for any state $\rho$ on $\lambda_\Gamma(\Gamma)'$, the composition $\rho\circ \varphi|_{\ll(\G)}$ is a $\Lambda$-invariant mean on $\ll(\G)$. Hence $\Lambda$ is amenable.
\endproof

\begin{thm}\label{pi-amen<->triv-acting}
Let $\pi$ be a unitary representation of $\Gamma$. 
A normal subgroup $\Lambda\trianglelefteq \G$ is $\pi$-amenable 
if and only if it acts trivially on $\cB_\pi$.
\end{thm}

\proof
Suppose $\lambda\trianglelefteq \G$ is $\pi$-amenable, and let 
$\varphi: B(\Hil_\pi) \to \pi(\Lambda)'$ be a $\Gamma$-map. 
By Proposition \ref{min-image-properties}, $\varphi|_{\cB_\pi}$ is an embedding. 
Since $\Lambda$ acts trivially on $\pi(\Lambda)'$, 
it follows that $\Lambda$ also acts trivially on $\cB_\pi$.

Conversely, suppose $\Lambda$ acts trivially on $\cB_\pi$. 
It follows that $\cB_\pi\subseteq \pi(\Lambda)'$. 
Hence there is a $\Gamma$-map $B(\Hil_\pi) \to \pi(\Lambda)'$, namely 
any $\Gamma$-projection $B(\Hil_\pi) \to \cB_\pi$.
\endproof

\begin{cor}\label{pi-amen-rad}
The kernel of the action $\G\act\cB_\pi$ is the unique 
maximal $\pi$-amenable normal subgroup of $\Gamma$, 
and it contains all $\pi$-amenable normal subgroups of $\Gamma$. 
\end{cor}

\begin{defn}\label{defn-rad_pi}
We call the kernel of the action of $\G\act\cB_\pi$ 
the \emph{$\pi$-amenable radical} of $\Gamma$, 
and denote it by $\Rad_\pi(\Gamma)$.
\end{defn}

\subsection{$\pi$-co-amenability}

\begin{defn}
If $\Lambda \trianglelefteq \G$ is a normal subgroup, say that $\Lambda$ is \emph{$\pi$-co-amenable} if there exists a $\G$-invariant state $\psi: \pi(\Lambda)' \to \C$.
\end{defn}

\begin{prop} \label{coamen implies pi-coamen}
Let $\pi:\G \to B(\Hil_\pi)$ be a unitary representation. Every co-amenable normal subgroup of $\G$ is $\pi$-co-amenable.
\end{prop}

\proof
Suppose $\Lambda$ is a co-amenable normal subgroup of $\G$. By normality, there is a canonical $\G/\Lambda$-action on $\pi(\Lambda)'$, and by co-amenability, there is a $\G$-map $\psi: \ell^\infty(\G/\Lambda) \to \C$. Since $\ell^\infty(\G/\Lambda)$ is $\G/\Lambda$-injective, there is a $\G/\Lambda$-map $\varphi: \pi(\Lambda)' \to \ell^\infty(\G/\Lambda)$, which is evidently also a $\G$-map. So $\psi \circ \varphi: \pi(\Lambda)' \to \C$ is a $\G$-map.
\endproof

\begin{prop}
Let $\pi:\G \to B(\Hil_\pi)$ be a unitary representation such that $\cB_\pi \cong C(\partial_F \G)$. Then a normal subgroup is co-amenable iff it is $\pi$-co-amenable.
\end{prop}

\proof
The forward direction follows from Proposition \ref{coamen implies pi-coamen}. For the converse, suppose $\Lambda$ is a normal $\pi$-co-amenable subgroup of $\G$, so that there exists a $\G$-map $\psi: \pi(\Lambda)' \to \C$. Since $\cB_\pi = \partial_F \G$, by Proposition \ref{prop:weak-contmnt-->map-bnds} there is a $\Gamma$-map $\varphi: B(\ell^2(\G)) \to  B(\Hil_\pi)$. The canonical embedding $\ell^\infty(\G/\Lambda) \hookrightarrow \ell^\infty(\G)$ identifies $\ell^\infty(\G/\Lambda)$ with a $\Gamma$-invariant $C^*$-subalgebra of $B(\ell^2(\G))$. Under this identification, for $f \in \ell^\infty(\G/\Lambda)$ and $s \in \Lambda$, using the fact that $\Lambda$ acts trivially on $\ell^\infty(\G/\Lambda)$, we have 
\[
\pi(s)\varphi(f) \pi(s^{-1}) = \varphi(\lambda_\Gamma(s)f\lambda_\Gamma(s^{-1})) = \varphi(f),
\]
which implies $\varphi$ maps $\ell^\infty(\G/\Lambda)$ into $\pi(\Lambda)' $. Hence, composing $\varphi$ with $\psi$ yields a $\G$-invariant state on $\ell^\infty(\G/\Lambda)$.
\endproof

\begin{cor}
A normal subgroup of $\G$ is co-amenable iff it is $\lambda_\G$-co-amenable.
\end{cor}

\begin{prop}\label{irred->coamen}
Let $\pi$ be a unitary representation of $\Gamma$, and 
let $\Lambda \trianglelefteq \G$ be a normal subgroup. 
If $\pi|_\Lambda$ is irreducible, then $\Lambda$ is $\pi$-co-amenable.
\end{prop}

\begin{proof}
If $\pi|_\Lambda$ is irreducible, then $\pi(\Lambda)' = \bC1$, and obviously 
there is a $\G$-invariant state on $\pi(\Lambda)'$.
\end{proof}

\begin{prop}\label{amen+coamen}
For a normal subgroup $\Lambda \trianglelefteq \G$, 
we have that $\pi$ is amenable if and only if $\Lambda$ is $\pi$-amenable and $\pi$-co-amenable.
\end{prop}

\proof
$(\Rightarrow)$ This follows from Proposition \ref{pi-amen-equiv-conds} and the obvious fact that a restriction of a $\G$-invariant state on $B(\Hil_\pi)$ to $\pi(\Lambda)'$ is a $\G$-invariant state on $\pi(\Lambda)'$.

$(\Leftarrow)$ Suppose $\Lambda$ is $\pi$-amenable, 
then by Theorem \ref{pi-amen<->triv-acting}, 
$\Lambda$ acts trivially on $\cB_\pi$, 
so that $\cB_\pi \subseteq \pi(\Lambda)'$. 
Thus any $\G$-invariant state on $\pi(\Lambda)'$ restricts to a 
$\G$-invariant state on $\cB_\pi$. 
Composing with a $\G$-idempotent 
$B(\Hil_\pi) \to \cB_\pi$ 
gives a $\G$-invariant state on $B(\Hil_\pi)$.
\endproof

\begin{cor} \label{norm subgp not pi-amen}
Let $\Lambda \trianglelefteq \G$ be a normal subgroup, and 
suppose $\pi$ is a non-amenable unitary representation of $\Gamma$ such that the 
restriction $\pi|_\Lambda$ is irreducible.
Then $\Lambda$ is not $\pi$-amenable.
\end{cor}

\begin{proof}
Since $\pi|_\Lambda$ is irreducible, $\Lambda$ is $\pi$-co-amenable, by 
Proposition \ref{irred->coamen}. 
Now, if $\Lambda$ is moreover $\pi$-amenable, then Proposition \ref{amen+coamen} 
implies amenability of $\pi$, which contradicts our assumption.
\end{proof}


\section{Extending boundary actions}
One of the main advantages of topological boundaries compared to their measurable counterparts is their extension properties. If $\Lambda$ is a normal subgroup of $\Gamma$, then the action $\Lambda\act \partial_F\Lambda$ extends to an action $\Gamma\act \partial_F\Lambda$. We prove a similar result in this section. Again, hidden in this result is a distinguished property of the regular representation, namely that it is the GNS representation of a pdf that is invariant under any automorphism.

\begin{defn}
Let $\pi$ be a unitary representation of $\Gamma$. An automorphism $\alpha$ of $\Gamma$ is called a \emph{$\pi$-automorphism} if $\alpha$ extends to a $*$-automorphism of the $C^*$-algebra $C^*_\pi(\Gamma)$. We denote by $\aut_\pi(\Gamma)$ the set of all $\pi$-automorphisms of $\Gamma$.
\end{defn}

\begin{prop}
We have $\aut_{\lambda_\Gamma}(\Gamma) = \aut(\Gamma)$.
\end{prop}

\begin{proof}
If $\alpha$ is an automorphism of $\Gamma$, then the GNS representation of the state $\tau(x) = \langle x \delta_e, \delta_e \rangle$ on $C_{\lambda \circ \alpha}^*(\Gamma)$ implements a $*$-isomorphism between $C_{\lambda \circ \alpha}^*(\Gamma)$ and $C^*_\lambda(\Gamma)$ such that $\lambda(\alpha(s)) \mapsto \lambda(s)$ for all $s \in \Gamma$.
\end{proof}

Note that $\alpha\in\aut_\pi(\Gamma)$ iff $\pi\circ\alpha$ is weakly equivalent to $\pi$, iff the automorphism on $C^*(\Gamma)$ induced by $\alpha$ leaves the ideal $I_\pi$ (i.e., the kernel of the canonical $*$-homomorphism $C^*(\Gamma) \to C^*_\pi(\Gamma)$) invariant. It is clear from the definition (or any of these equivalent descriptions) that $\aut_\pi(\Gamma)$ is a subgroup of $\aut(\Gamma)$, and that $\aut_\pi(\Gamma) = \aut_{\pi\circ\alpha}(\Gamma)$ for all $\alpha\in \aut_\pi(\Gamma)$.

Obviously any inner automorphism is a $\pi$-automorphism for any representation $\pi$. Thus, via inner automorphisms, $\Gamma/\cZ(\Gamma)$ is identified with a normal subgroup of $\aut_\pi(\Gamma)$. It follows from Proposition \ref{finconjclass-trivialaction} that $\cZ(\Gamma)$ acts trivially on the FH-boundary of any representation $\pi$, hence the action $\Gamma\act \cB_\pi$ factors through $\Gamma/\cZ(\Gamma)$.

\begin{thm}\label{thm:ext-bnd-act}
Let $\pi$ be a unitary representation of $\Gamma$. Then the action $\Gamma\act \cB_\pi$ extends to an action $\aut_\pi(\Gamma)\act \cB_\pi$. 
\end{thm}

\begin{proof}
Let $\pi: \Gamma\to\cU(\cH_\pi)$ be a unitary representation and let $\alpha\in \aut_\pi(\Gamma)$. Let ${\pi_\alpha}: \Gamma\to\cU(\cH_{\pi_\alpha})$ be the unitary representation where $\cH_{\pi_\alpha} = \cH_\pi$ and ${\pi_\alpha}(g) = \pi(\alpha(g))$ for $g\in \Gamma$. Then $\pi_\alpha$ is weakly equivalent to $\pi$, and we denote by $\alpha$ the $*$-automorphism on $C^*_\pi(\Gamma)= C^*_{\pi_\alpha}(\Gamma)$ which extends $\pi(g) \mapsto {\pi_\alpha}(g)$. Obviously a subspace of $B(\cH_\pi)$ is $\pi$-invariant iff it is ${\pi_\alpha}$-invariant, and any $\Gamma$-map on $B(\cH_\pi)$ is a $\Gamma$-map on $B(\cH_{\pi_\alpha})$, and vice versa. Hence, it follows that $\cB_\pi$ considered as a subspace of $B(\cH_{\pi_\alpha})$ with the action of $\Gamma$ via ${\pi_\alpha}$ is the FH-boundary of ${\pi_\alpha}$. Since $\pi$ and ${\pi_\alpha}$ are weakly equivalent, by Corollary \ref{cor:bndy in inj env} there is a $*$-automorphism $\Phi_\alpha:\cB_\pi \to \cB_{\pi_\alpha} = \cB_\pi$ which satisfies 
\begin{equation}\label{equiv-ext}
\Phi_\alpha(\pi(g) a \pi(g^{-1})) = \pi_\alpha(g) \Phi_\alpha(a) \pi_\alpha(g^{-1})= \pi(\alpha(g)) \Phi_\alpha(a) \pi(\alpha(g^{-1}))
\end{equation}
for all $g\in\Gamma$ and $a\in \cB_\pi$. Then 
\[\begin{split}
\Phi_{\alpha^{-1}}\circ\Phi_\alpha(\pi(g) a \pi(g^{-1})) 
&= 
\Phi_{\alpha^{-1}}({\pi_\alpha}(g) \Phi_\alpha(a) {\pi_\alpha}(g^{-1}))
\\&= 
\Phi_{\alpha^{-1}}(\pi(\alpha(g)) \Phi_\alpha(a) \pi(\alpha(g^{-1})))
\\&= 
\pi_{\alpha^{-1}}(\alpha(g)) \Phi_{\alpha^{-1}}(\Phi_\alpha(a)) \pi_{\alpha^{-1}}(\alpha(g^{-1}))
\\&= 
\pi(\alpha^{-1}(\alpha(g))) \Phi_{\alpha^{-1}}(\Phi_\alpha(a)) \pi(\alpha^{-1}(\alpha(g^{-1}))
\\&= 
\pi(g) \Phi_{\alpha^{-1}}\circ\Phi_\alpha(a) \pi(g^{-1}) ,
\end{split}\]
which shows $\Phi_{\alpha^{-1}}\circ\Phi_\alpha$ is a $\Gamma$-map on $\cB_\pi$, hence identity by $\pi$-rigidity. Thus, $\Phi_{\alpha^{-1}}= (\Phi_\alpha)^{-1}$. If $\alpha, \beta \in \aut_\pi(\Gamma)$, then straightforward calculations similar to the above shows $\Phi_{\beta^{-1}}\circ\Phi_{\alpha^{-1}}\circ\Phi_{\alpha\beta}$ is a $\Gamma$-map on $\cB_\pi$, which implies $\Phi_\beta\circ\Phi_\alpha = \Phi_{\alpha\beta}$. Hence $\alpha\mapsto \Phi_\alpha$ defines an action $\aut_\pi(\Gamma)\act \cB_\pi$. Now let $g\in \Gamma$, and let $\alpha_g = \Ad_g$ be the corresponding inner automorphism on $\Gamma$. Then for any $h\in \Gamma$ and $a\in \cB_\pi$ we have
\[\begin{split}
\Phi_{\alpha_{g}}\circ \Ad_{\pi(g^{-1})} (\pi(h) a \pi(h^{-1})) 
&= 
\Phi_{\alpha_{g}} (\pi(g^{-1}h) a \pi(h^{-1}g))
\\&= 
\pi(\alpha_{g}(g^{-1}h)) \Phi_{\alpha_{g}} (a) \pi(\alpha_{g}(h^{-1}g))
\\&= 
\pi(hg^{-1}) \Phi_{\alpha_{g}} (a) \pi(gh^{-1})
\\&= 
\pi(h) \pi(g^{-1}) \Phi_{\alpha_{g}} (a) \pi(g)\pi(h^{-1})
\\&= 
\pi(h) \pi(\alpha_{g}(g^{-1})) \Phi_{\alpha_{g}} (a) \pi(\alpha_{g}(g))\pi(h^{-1})
\\&= 
\pi(h)\Phi_{\alpha_{g}} (\pi(g^{-1}) a \pi(g))\pi(h^{-1})
\\&= 
\pi(h)\Phi_{\alpha_{g}}\circ \Ad_{\pi(g^{-1})} (a)\pi(h^{-1}) ,
\end{split}\]
which shows $\Phi_{\alpha_{g}}\circ \Ad_{\pi(g^{-1})}$ is a $\Gamma$-map on $\cB_\pi$, hence identity. Hence $\Phi_{\alpha_{g}} = \Ad_{\pi(g)}$.
\end{proof}

Note that if $\alpha\mapsto \Phi_\alpha, \Psi_\alpha$ are two actions of $\aut_\pi(\Gamma)$ on $\cB_\pi$, both extending the $\Gamma$-action and satisfy \eqref{equiv-ext}, then $\Phi_{\alpha^{-1}} \circ\Psi_\alpha$ is a $\Gamma$-map on $\cB_\pi$ for all $\alpha$, which implies $\Phi_{\alpha^{-1}} = \Psi_\alpha$. Thus, there is a unique action $\aut_\pi(\Gamma)\act \cB_\pi$ that extends the $\Gamma$-action and satisfies \eqref{equiv-ext}.

Next, we identify the kernel of this action. Recall from Definition \ref{defn-rad_pi} that $\Rad_\pi(\Gamma)$ is the kernel of the action $\Gamma \curvearrowright \cB_\pi$.

\begin{lem}
The $\pi$-amenable radical $\Rad_\pi(\Gamma)$ is invariant under any $\pi$-automorphism.
\end{lem}

\begin{proof}
Let $g\in \Rad_\pi(\Gamma)$ and $\alpha\in \aut_\pi(\Gamma)$. Then by \eqref{equiv-ext} we have
\[
\Phi_\alpha(a) = \Phi_\alpha(\pi(g) a \pi(g^{-1})) = \pi(\alpha(g)) \Phi_\alpha(a) \pi(\alpha(g^{-1}))
\]
for all $a\in \cB_\pi$, which implies $\alpha(g)\in \Rad_\pi(\Gamma)$.
\end{proof}

Hence, by the above we obtain a group homomorphism 
\[
\kappa : \aut_\pi(\Gamma) \to \aut(\Gamma/\Rad_\pi(\Gamma)) .
\]

\begin{prop}
Let $\pi$ be a unitary representation of $\Gamma$, and let $\aut_\pi(\Gamma)\act \cB_\pi$ be the unique extension of the $\Gamma$-action which satisfies \eqref{equiv-ext}. Then $$\ker(\aut_\pi(\Gamma)\act \cB_\pi) = \ker(\kappa).$$
\end{prop}

\begin{proof}
Let $\alpha\in \ker(\kappa)$, then $g^{-1}\alpha(g) \in \Rad_\pi(\Gamma)$ for all $g\in \Gamma$, and so by \eqref{equiv-ext} we have
\[
\Phi_\alpha(\pi(g) a \pi(g^{-1})) = \pi(\alpha(g)) \Phi_\alpha(a) \pi(\alpha(g^{-1})) =  \pi(g) \Phi_\alpha(a) \pi(g^{-1}),
\]
which implies $\Phi_\alpha$ is $\Gamma$-equivariant, hence the identity, and so in the kernel of the action $\aut_\pi(\Gamma)\act \cB_\pi$.

Conversely, suppose $\alpha\in\aut_\pi(\Gamma)$ acts trivially on $\cB_\pi$. Then again using \eqref{equiv-ext} we have
\[\begin{split}
\pi(g) \,a\, \pi(g^{-1}) &= \Phi_\alpha(\pi(g) a \pi(g^{-1})) \\&= \pi(\alpha(g)) \Phi_\alpha(a) \pi(\alpha(g^{-1})) \\&= \pi(\alpha(g)) \,a\, \pi(\alpha(g^{-1}))
\end{split}\]
for all $g\in\Gamma$ and $a\in \cB_\pi$, which implies $g^{-1}\alpha(g) \in \Rad_\pi(\Gamma)$, hence $\alpha\in \ker(\kappa)$.
\end{proof}

\begin{cor}\label{cor:faithful-act-->faithful-ext}
If the action $\Gamma\act \cB_\pi$ is faithful then so is the extended action $\aut_\pi(\Gamma)\act \cB_\pi$. 
\end{cor}

\begin{thm}
Let $\pi$ be a representation of $\Gamma$ and $\Lambda\leq \Gamma$ be a normal subgroup. Then the action $\Lambda\act \cB_{\pi|_\Lambda}$ extends to an action $\Gamma\act \cB_{\pi|_\Lambda}$. If the $\Lambda$-action is faithful, then $$\ker(\Gamma\act \cB_{\pi|_\Lambda}) = C_{\pi(\Gamma)}(\pi(\Lambda)) ,$$ where $C_{\pi(\Gamma)}(\pi(\Lambda))$ is the centralizer of $\pi(\Lambda)$ in $\pi(\Gamma)$.
\end{thm}

\begin{proof}
Since $\Lambda$ is normal in $\Gamma$ we have $\Ad_{\pi(g)}(\pi(\Lambda))= \pi(\Lambda)$ for every $g\in \Gamma$, which implies the $C^*$-algebra $C^*_{\pi|_\Lambda}(\Lambda)$ is $\Gamma$-invariant and $\Ad_{\pi(g)}\in \aut_{\pi|_\Lambda}(\Lambda)$ for all $g\in \Gamma$. Thus, this yields a group homomorphism from $\Gamma$ to $\aut_{\pi|_\Lambda}(\Lambda)$. The kernel of this homomorphism is $C_{\pi(\Gamma)}(\pi(\Lambda))$. By Corollary \ref{cor:faithful-act-->faithful-ext}, since the action $\Lambda \act \cB_{\pi|_\Lambda}$ is faithful, the action of the quotient group ${\Gamma}/{C_{\pi(\Gamma)}(\pi(\Lambda))}$ on $\cB_{\pi|_\Lambda}$ is faithful. Hence it follows $\ker(\Gamma\act \cB_{\pi|_\Lambda}) = C_{\pi(\Gamma)}(\pi(\Lambda))$.
\end{proof}


\section{Faithfulness and uniqueness of trace}

One of the most important applications of boundary actions was in settling the problem of characterizing groups with the unique trace property. In fact, it was proved in \cite[Theorem 4.1]{BKKO} that every trace on $C^*_\lambda(\Gamma)$ is supported on $\Rad(\Gamma)$, the amenable radical of $\Gamma$, which as mentioned before, is the kernel of the action $\Gamma\act \partial_F \Gamma$. In this section we prove generalizations of this result in the case of representations with commutative FH-boundaries (e.g., quasi-regular representations).

But note that the regular representation admits several special and universal properties that are unique to that case. So, in order to avoid any restrictive assumptions, first, we need to rephrase the above result in the form of an equivalent statement which does not hide any of those properties.

Observe that a trace $\tau$ on $C^*_\lambda(\Gamma)$ is nothing but a $\Gamma$-map $\tau: C^*_\lambda(\Gamma)\to \bC$, where $\Gamma\act\bC$ is the trivial action. By $\Gamma$-injectivity, any such map extends to a $\Gamma$-map $\tilde\tau: B(\ell^2(\Gamma)) \to \cI_\Gamma(\bC) = C(\partial_F \Gamma) = \cB_\lambda$. Hence, the property that every trace on $C^*_\lambda(\Gamma)$ is supported on $\Rad(\Gamma) = \ker(\Gamma\act \cB_\lambda)$ is equivalent to the property that every $\Gamma$-extension $\tilde\tau: B(\ell^2(\Gamma)) \to \cI_\Gamma(\bC) = C(\partial_F \Gamma) = \cB_\lambda$ of any trace $\tau$ on $C^*_\lambda(\Gamma)$ vanishes on all $g\notin \ker(\Gamma\act \cB_\lambda)$. We see below that the fact that any trace on $C^*_\lambda(\Gamma)$ can be extended to a $\Gamma$-map on $B(\ell^2(\Gamma))$ is the key in the above result.

\begin{defn}
Let $\pi$ be a unitary representation of $\Gamma$. A trace $\tau$ on the $C^*$-algebra $C^*_\pi(\Gamma)$ is called an \emph{extendable trace} if it extends to a $\Gamma$-map $B(\cH_\pi) \to \cB_\pi$.
\end{defn}

One can see that if $\pi$ is weakly contained in the regular representation, then by Theorem \ref{thm:exact--w-cntmnt-reg<-->bnd-Furs} any trace $\tau$ on $C^*_\pi(\Gamma)$ is an extendable trace. In fact, if $\pi$ is a representation such that $\cB_\pi=C(\partial_F\Gamma)$, then by $\Gamma$-injectivity and the fact $C(\partial_F\Gamma)=\cI_\Gamma(\bC)$ it follows every trace on $C^*_\pi(\Gamma)$ is extendable.

\begin{lem}\label{lem:unq-tr}
Let $\pi$ be a unitary representation of $\Gamma$ such that $\cB_\pi$ is commutative. If $\psi: B(\cH_\pi) \to \cB_\pi$ is a $\Gamma$-map and $g\notin \Rad_\pi(\Gamma)$, then $\psi(\pi(g))$ vanishes on some state on $\cB_\pi$.
\end{lem}

\begin{proof}
Since $\cB_\pi$ is commutative, it is of the form $C(Z)$ for some compact $\Gamma$-space $Z$. Now let $\psi: B(\cH_\pi) \to C(Z)$ be a $\Gamma$-map and let $g\notin \ker(\Gamma\act Z)$. So there exist $z\in Z$ and positive $f\in C(Z)$ such that $f(z) = 1$ and $f(g^{-1}z) = 0$. Hence, it follows from \cite[Lemma 2.2]{HartKal} that $\delta_z\circ\psi(\pi(g)) = 0$.
\end{proof}

\begin{thm}\label{thm:unq-ext-trc}
Let $\pi$ be a unitary representation of $\Gamma$ such that $\cB_\pi$ is commutative. Then any extendable trace on $C^*_\pi(\Gamma)$ is supported on $\Rad_\pi(\Gamma)$.
\end{thm}
 
\begin{proof}
Let $\tau$ be an extendable trace on $C^*_\pi(\Gamma)$ and let $\tilde\tau: B(\cH_\pi) \to \cB_\pi$ be a $\Gamma$-map that extends $\tau$. Then for any $g\in \Gamma$, $\tilde\tau(\pi(g))$ is a constant function on the Gelfand spectrum $Z$ of $\cB_\pi$, taking the value $\tau(\pi(g))$. If $g\notin \ker(\Gamma\act Z)$, then by Lemma \ref{lem:unq-tr}, $\tilde\tau(\pi(g))$ vanishes on some probability on $Z$. Hence, it follows that $\tau(\pi(g))=0$.
\end{proof}

\begin{cor}
Let $\pi$ be a unitary representation of $\Gamma$ whose FH-boundary is commutative and faithful. Then either $C^*_\pi(\Gamma)$ admits no extendable trace, or otherwise $\pi$ weakly contains $\lambda_\Gamma$ and the canonical trace is the unique extendable trace on $C^*_\pi(\Gamma)$.
\end{cor}

\begin{cor}\label{cor:unq-tr}
Suppose $\G$ has trivial amenable radical. Let $\pi$ be a unitary representation of $\Gamma$ such that $\cB_\pi=C(\partial_F\G)$. Then either $C^*_\pi(\Gamma)$ admits no trace, or otherwise $\pi$ weakly contains $\lambda_\Gamma$ and the canonical trace is the unique trace on $C^*_\pi(\Gamma)$.
\end{cor}
\begin{proof}
Since $\cB_\pi=C(\partial_F\G)$, by comments above Lemma \ref{lem:unq-tr}, every trace on $C^*_\pi(\Gamma)$ is extendable. Since $\G$ has trivial amenable radical, the action $\G\act\partial_F\G$ is faithful. Hence the claim follows from Theorem \ref{thm:unq-ext-trc}.
\end{proof}

\begin{example}
Suppose $\Gamma$ has trivial amenable radical and that the centralizer $C_\G(g)$ of every non-trivial element $g\in \G$ is amenable (e.g. the free group $\bF_n$, $n\geq2$, satisfies this). Let $\rho_\G:\G\to B(\ell^2(\G))$ be the right regular representation. Let $\pi:\G\to B(\ell^2(\G\setminus\{e\}))$ be the unitary representation $\pi(g):=\lambda_\G(g)\rho_\G(g)$. This is in fact the Koopman representation associated to the action $\G\act S=\G\setminus\{e\}$ by conjugations, which is also unitarily equivalent to the direct sum of quasi-regular representations associated to stabilizer subgroups. Note that for $g\in S$, the stabilizer of $g$ in $\G$ is the centralizer $C_\G(g)$ of $g$ in $\G$, hence amenable by the assumption. Therefore the quasi-regular representation associated to stabilizer subgroups are weakly contained in the regular representation $\lambda_\G$. Hence, $\pi$ is weakly contained in $\lambda_\G$. Moreover, as seen in Example \ref{ex:koop-dis}, we have $\cB_\pi = C(\partial_F\G)$. Thus, by Corollary \ref{cor:unq-tr}, $C^*_\pi(\Gamma)$ either admits no trace, or otherwise $\pi$ is weakly equivalent to $\lambda_\Gamma$ and the canonical trace is the unique trace on $C^*_\pi(\Gamma)$. In particular, if $\G$ is $C^*$-simple (e.g. $\bF_n$, $n\geq2$), then the latter case holds.
\end{example}

\section{Quasi-regular representations}

In this section we study the case of quasi-regular representations. 
We see in this case, the FH-boundary is commutative, hence 
yields a compact $\Gamma$-space 
which should be considered as the ``\emph{Furstenberg boundary of the quotient}''.

Throughout this section $\Lambda\leq \Gamma$ is a subgroup, 
and $\Gamma\act \Gamma/\Lambda$ is the 
canonical action on the set of left $\Lambda$-cosets.

We denote by $\lambda_{\Gamma/\Lambda} : \Gamma \to B(\ell^2(\Gamma/\Lambda))$ 
the quasi-regular representation defined by 
$\lambda_{\Gamma/\Lambda}(g) \delta_{g'\Lambda} = \delta_{gg'\Lambda}$.

Observe that the map 
\[
B(\ell^2(\Gamma/\Lambda)) \ni T \mapsto f_T \in \ell^\infty(\Gamma/\Lambda) \ \ ; \ \
f_T(g\Lambda) := \langle T \delta_{g\Lambda} , \delta_{g\Lambda} \rangle
\]
is a $\Gamma$-projection.
Thus, $\cB_{\lambda_{\Gamma/\Lambda}} \subseteq  \ell^\infty(\Gamma/\Lambda)$, 
which implies $\cB_{\lambda_{\Gamma/\Lambda}}$ is commutative as it 
inherits its Choi-Effros product from $\ell^\infty(\Gamma/\Lambda)$. 

Hence we have $\cB_{\lambda_{\Gamma/\Lambda}} = C(\partial_{FH} (\Gamma/\Lambda))$, 
where $\partial_{FH} (\Gamma/\Lambda)$ is a compact $\Gamma$-space.

\begin{defn}
The compact $\Gamma$-space $\partial_{FH} (\Gamma/\Lambda)$ 
will be called the \emph{FH-boundary} of the pair $(\Gamma, \Lambda)$.
We say $\Gamma\act X$ is a (topological) 
$\Gamma/\Lambda$-boundary action, or that $X$ is a (topological) 
$\Gamma/\Lambda$-boundary, if there is a $\Gamma$-map 
$C(X) \to B(\ell^2(\Gamma/\Lambda))$ and every such map 
is isometric. 
\end{defn}

Obviously if $X$ is a $\Gamma/\Lambda$-boundary then 
$C(X)\subseteq C(\partial_{FH} (\Gamma/\Lambda))$ as a $\Gamma$-invariant subspace.

\begin{prop}
If $\Lambda$ is normal in $\Gamma$ then 
$\partial_{FH} (\Gamma/\Lambda) = \partial_F (\Gamma/\Lambda)$.
\end{prop}

\begin{proof}
This follows immediately from construction.
\end{proof}

\begin{lem}\label{co-amen-subgroups}
Let $\Lambda_1\leq \Lambda_2\leq \Gamma$ be subgroups.
If $\Lambda_1$ is co-amenable in $\Lambda_2$ then 
$\partial_{FH} (\Gamma/\Lambda_1) \cong \partial_{FH} (\Gamma/\Lambda_2)$ 
as $\Gamma$-spaces. 
\end{lem}

\begin{proof}
Suppose $\Lambda_1$ is co-amenable in $\Lambda_2$, i.e., 
$1_{\Lambda_2} \prec \lambda_{\Lambda_2/\Lambda_1}$. Then 
\[\begin{split}
\lambda_{\Gamma/\Lambda_2} 
&= 
\operatorname{Ind}_{\Lambda_2}^\Gamma(1_{\Lambda_2}) 
\prec 
\operatorname{Ind}_{\Lambda_2}^\Gamma(\lambda_{\Lambda_2/\Lambda_1})
\\&=
\operatorname{Ind}_{\Lambda_2}^\Gamma(\operatorname{Ind}_{\Lambda_1}^{\Lambda_2 }(1_{\Lambda_1})
=
\operatorname{Ind}_{\Lambda_1}^\Gamma(1_{\Lambda_1}) 
=
\lambda_{\Gamma/\Lambda_1}.
\end{split}\]
Thus, by Proposition \ref{prop:weak-contmnt-->map-bnds}, there is a $\Gamma$-map 
$\psi: C(\partial_{FH} (\Gamma/\Lambda_1)) \to C(\partial_{FH} (\Gamma/\Lambda_2))$. 
Note that $\ell^\infty(\Gamma/\Lambda_2)$ is canonically identified with a $\Gamma$-invariant 
$C^*$-subalgebra of $\ell^\infty(\Gamma/\Lambda_1)$.
Hence we get the string of $\Gamma$-maps
\[\begin{split}
C(\partial_{FH} (\Gamma/\Lambda_1)) 
&\xrightarrow{\psi} 
C(\partial_{FH} (\Gamma/\Lambda_2))
\xrightarrow{\id} 
\ell^\infty(\Gamma/\Lambda_2)
\\&\xrightarrow{\id} 
\ell^\infty(\Gamma/\Lambda_1)
\xrightarrow{\varphi_1} 
C(\partial_{FH} (\Gamma/\Lambda_1)), 
\end{split}\]
where $\varphi_1$ is an idempotent $\Gamma$-map.
By Proposition \ref{min-image-properties} the composition of the above maps 
is the identity map on $C(\partial_{FH} (\Gamma/\Lambda_1))$. 
Hence it follows that $\psi$ induces a $\Gamma$-equivariant 
homeomorphism between $\partial_{FH} (\Gamma/\Lambda_1)$ 
and $\partial_{FH} (\Gamma/\Lambda_2)$.
%
\end{proof}

\begin{cor}
The FH-boundary of the the pair $(\Gamma, \Lambda)$ 
is trivial if and only if 
 $\Lambda\leq \Gamma$ is co-amenable.
\end{cor}

\begin{proof}
Setting $\Lambda_1 = \Lambda$ and $\Lambda_2 = \Gamma$ in 
Lemma \ref{co-amen-subgroups} yields the result.
\end{proof}

\begin{cor}
We have $\partial_{FH} (\Gamma/\Lambda) = \partial_F (\Gamma)$ 
as $\Gamma$-spaces if and only if $\Lambda$ is amenable.
\end{cor}

\begin{proof}
Setting $\Lambda_1 = \{e\}$ and $\Lambda_2 = \Lambda$ in 
Lemma \ref{co-amen-subgroups} yields the result.
\end{proof}

\subsection{Dynamical characterization of $\partial_{FH} (\Gamma/\Lambda)$}

In this section we prove a dynamical characterization of 
$\Gamma/\Lambda$-boundaries 
in terms of contractibility of certain measures.

\begin{defn}
A probability measure $\nu$ on a compact $\Gamma$-space $X$ 
is \emph{contractible} if $\delta_x\in \overline{\Gamma\nu}^{\text{weak*}}$ for every $x\in X$.
We denote by $\pr^c(X)$ the set of all contractible probabilities on $X$.
\end{defn}

We also denote by $\pr_\Lambda(X)$ the set of all $\Lambda$-invariant probabilities on $X$.

\begin{thm}[Azencott]
A probability measure $\nu$ on a compact $\Gamma$-space $X$ 
is \emph{contractible} iff its corresponding Poisson map 
$\cP_\nu$ is isometric.
\end{thm}

\begin{thm}\label{dyn-charcterization}
Let $\Lambda\leq \Gamma$ be a subgroup. An action $\Gamma\act X$ is a $\Gamma/\Lambda$-boundary iff $\emptyset \neq \pr_\Lambda(X) \subseteq \pr^c(X)$.
\end{thm}

\begin{proof}
Suppose $\Gamma\act X$ is a $\Gamma/\Lambda$-boundary. 
So there is a $\Gamma$-embedding $C(X) \hookrightarrow \ell^\infty(\Gamma/\Lambda)$. The restriction of the state $\delta_\Lambda\in \ell^1(\Gamma/\Lambda)$ yields a $\Lambda$-invariant probability on $X$. Moreover, if $\eta\in\pr_\Lambda(X)$, then the Poisson map 
$\cP_\eta$ is a $\Gamma$-map from $C(X)$ to $\ell^\infty(\Gamma/\Lambda)$, hence isometric by definition of $\Gamma/\Lambda$-boundaries, which implies $\eta\in\pr^c(X)$.

Conversely, suppose $\emptyset \neq \pr_\Lambda(X) \subseteq \pr^c(X)$. Let $\eta\in\pr_\Lambda(X)$. Then $\cP_\eta$ is a $\Gamma$-map from $C(X)$ to $\ell^\infty(\Gamma/\Lambda)$, and 
any such map is the Poisson map corresponding to a $\Lambda$-invariant measure on $X$. Since all such measures are contractible, it follows that any $\Gamma$-map from $C(X)$ to $\ell^\infty(\Gamma/\Lambda)$ is isometric. Hence $X$ is a $\Gamma/\Lambda$-boundary.
\end{proof}

\begin{cor}\label{inv-prob-->subgrp-bnd}
Let $\Lambda_1\leq\Lambda_2\leq\Gamma$ be subgroups, and let  
$X$ be a $\Gamma/\Lambda_1$-boundary. 
Then $X$ is a $\Gamma/\Lambda_2$-boundary 
if and only if it admits a $\Lambda_2$-invariant probability. 
\end{cor}

\begin{proof}
Since $X$ is a $\Gamma/\Lambda_1$-boundary, every $\Lambda_1$-invariant probability $\eta\in\pr(X)$ is contractible by Theorem \ref{dyn-charcterization}. Hence, $$\pr_{\Lambda_2}(X) \subseteq \pr_{\Lambda_1}(X) \subseteq \pr^c(X).$$ Thus, again Theorem \ref{dyn-charcterization} implies that $X$ is a $\Gamma/\Lambda$-boundary if and only if it admits a $\Lambda_2$-invariant probability.
\end{proof}

An important case is when $\Lambda_1$ is trivial in the above corollary, which gives the following.

\begin{cor}\label{inv-prob-->bnd}
Let $\Lambda$ be a subgroup of $\Gamma$. Then a boundary action $\Gamma\act X$ is also a $\Gamma/\Lambda$-boundary action if and only if $X$ admits a $\Lambda$-invariant probability. 
\end{cor}

\begin{example}
Let $p>1$ be a prime, $n\geq2$, and let $G= SL_n(\bQ_p)$. With the $p$-adic topology, $G$ is a locally compact group, containing $\Gamma = SL_n(\bZ[\frac1p])$ as a dense subgroup. The closure of the subgroup $\Lambda= SL_n(\bZ)$ in $G$ is the compact subgroup $SL_n(\bZ_p)$. Since $\Gamma$ is dense in $G$ it follows that $\partial_F G$ is a $\Gamma$-boundary. Since $\Lambda$ is precompact in $G$, $\partial_F G$ admits a $\Lambda$-invariant probability. Hence, by Corollary \ref{inv-prob-->bnd}, $\partial_F G$ is a $\Gamma/\Lambda$-boundary.
\end{example}

The above example is indeed a special case of the following.

A subgroup $\Lambda\leq \Gamma$ is said be \emph{commensurated} in 
$\Gamma$ if $g\Lambda g^{-1} \cap \Lambda$ has finite index in $\Lambda$ 
for all $g\in \Gamma$.

Let $\operatorname{Symm}(\Gamma/\Lambda)$ denote the group of all 
symmetries of the set $\Gamma/\Lambda$. With the pointwise convergence 
topology, $\operatorname{Symm}(\Gamma/\Lambda)$ is a topological 
group, and the canonical action of $\Gamma$ on $\Gamma/\Lambda$ 
gives a group homomorphism 
$\Gamma\to \operatorname{Symm}(\Gamma/\Lambda)$. 
We denote by $\overline{\Gamma}$ 
the closure of image of $\Gamma$ under this 
homomorphism (it will always be clear from the context 
with respect to which subgroup $\Lambda$ this completion is taken). 
Then $\overline{\Lambda}$ is a compact open subgroup of 
$\overline{\Gamma}$, and 
$\overline{\Gamma}/\overline{\Lambda} = \Gamma/\Lambda$.

\begin{prop}\label{commensurated->bnd}
Suppose $\Lambda\leq\Gamma$ is commensurated. Then $\partial_F\overline{\Gamma}$ 
(the Furstenberg boundary of the topological group $\overline{\Gamma}$) 
is a $\Gamma/\Lambda$-boundary. 
\end{prop}

\begin{proof}
Since $\Gamma$ is dense in $\overline{\Gamma}$, it follows that 
$\partial_F\overline{\Gamma}$ 
is a $\Gamma$-boundary. 
Since $\overline{\Lambda}$ is compact, it fixes some $\eta\in\pr(\partial_F\overline{\Gamma})$. 
Thus, Corollary \ref{inv-prob-->bnd} implies  $\partial_F\overline{\Gamma}$ 
is a $\Gamma/\Lambda$-boundary.
\end{proof}

\section{Non-commutative examples}

So far, all the examples we have given of FH-boundaries are commutative $C^*$-algebras. This is not always the case though:

\begin{theorem}
Every separable purely infinite unital $C^*$-algebra $\cA$ 
appears as a subalgebra of a $\pi$-boundary for some representation 
$\pi$ of some countable group $\Gamma$.
\end{theorem}

\begin{proof}
Let $\Hil$ be a separable infinite-dimensional Hilbert space such that $\cA\subseteq B(\cH)$. Let $\{p_n : n\in\bN\}$ be a sequence of nonzero projections that generate $\cA$. Since $\cA$ is purely infinite, for any $n\in\bN$, both $p_n$ and $1-p_n$ are infinite projections. For each $n\in\bN$ choose a unitary $u_n$ such that $u_np_nu_n^* = 1-p_n$. Let $\Lambda_n$ be a countable group of unitaries that generates the von Neumann algebra of operators on $\cH$ which commute with $p_n$. Moreover, choose projections $q^{(n)}_1$ and $q^{(n)}_2$ and unitaries $v_n$ and 
$w_n$ such that $q^{(n)}_1+q^{(n)}_2 = p_n$, $v_nq^{(n)}_1v_n^* = p_n$ and $w_nq^{(n)}_2w_n^* = p_n$.
Let $\Gamma$ be the group generated by $\cup_n\Lambda_n$ and $\{u_n, v_n, w_n :n\in\bN\}$.

Suppose $\Phi:B(\cH)\to B(\cH)$ is a $\Gamma$-projection. For any $n\in\bN$ we show $\Phi(p_n)=p_n$. So, fix $n\in \bN$. First, observe that since $u\Phi(p_n)u^* = \Phi(u p_n u^*) = \Phi(p_n)$ for all $u\in \Lambda_n$, we have $\Phi(p_n)\in \{p_n\}'' = \bC p_n\oplus\bC (1-p_n)$. Let $\Phi(p_n) = \alpha p_n + \beta (1-p_n)$ for $\alpha, \beta\in\bC$. Then $\Phi(q^{(n)}_1) = \Phi(v_n^*p_nv_n) = v_n^*\Phi(p_n)v_n = \alpha q^{(n)}_1 + \beta (1-q^{(n)}_1)$, and similarly (using the $w_n's$), $\Phi(q_2^{(n)}) = \alpha q_2^{(n)} + \beta (1-q_2^{(n)})$. Therefore \[\begin{split}\alpha p_n + \beta (1-p_n) &= \Phi(p_n) = \Phi(q^{(n)}_1) + \Phi(q^{(n)}_2) \\&= \alpha q^{(n)}_1 + \beta (1-q^{(n)}_1) + \alpha q^{(n)}_2 + \beta (1-q^{(n)}_2)\\&=\alpha p_n + \beta (2-p_n) ,\end{split}\]
which implies $\beta = 0$. Thus, $\alpha p_n = \Phi(p_n) = \Phi(\Phi(p_n)) = \alpha \Phi(p_n) = \alpha^2 p_n$. So $\alpha = 0$ or $\alpha = 1$. If $\alpha=0$, then $0=u_n \Phi(p_n)u_n^* = \Phi(1-p_n) = I - \Phi(p_n)$, which contradicts the fact that $p_n$ and $1-p_n$ are unitarily equivalent. This proves the claim.

Now let $\pi:\Gamma\to \cU(\cH)$ be the identity representation, and let $E: B(\cH)\to \cB_\pi$ be a $\Gamma$-projection. Then by the above we have $E|_\cA = \id$, which implies $\cA\subseteq \cB_\pi$. Moreover, since the multiplication of $\cB_\pi$ is defined by the Choi-Effros product induced by $E$, it is seen that the multiplication in $\cA$ coincides with that inherited from $\cB_\pi$. 
\end{proof}

\end{document}